\documentclass[a4paper,11pt,abstract=on]{scrartcl}

\usepackage{amssymb}
\usepackage{amsmath}
\usepackage{amsthm}

\usepackage{mathtools}
\usepackage{etextools}
\usepackage{ifthen}

\usepackage{bbm}

\usepackage{caption}
\usepackage{pgfplots}

\usepackage{algorithm}
\usepackage[noend]{algpseudocode}

\usepackage{hyperref}

\usepackage{booktabs}
\usepackage{multirow}
\usepackage{multicol}
\mathtoolsset{showonlyrefs}

\usepackage{threeparttable}

\usepackage{nicefrac,xfrac}

\usepackage{tikz}
\definecolor{darkolivegreen}{rgb}{0.33, 0.42, 0.18}

\newcommand{\lineblue}{%
  \tikz[baseline=-0.75ex] \draw[-, blue, very thick] (0,0) -- (1,0);%
}
\newcommand{\linepurple}{%
  \tikz[baseline=-0.75ex] \draw[-, purple!70!black, very thick] (0,0) -- (1,0);%
}
\newcommand{\lineorange}{%
  \tikz[baseline=-0.75ex] \draw[-, orange, very thick] (0,0) -- (1,0);%
}
\newcommand{\linegreen}{%
  \tikz[baseline=-0.75ex] \draw[-, darkolivegreen, very thick] (0,0) -- (1,0);%
}

\newcommand{\rv}[1]{{#1}}

\newcommand{\EE}{\mathbb{E}}
\newcommand{\PP}{\mathbb{P}}
\newcommand{\RR}{\mathbb{R}}
\newcommand{\NN}{\mathbb{N}}
\newcommand{\CC}{\mathbb{C}}

\newcommand{\sphere}{\mathbb{S}}

\newcommand{\cN}{\mathcal{N}}

\newcommand{\cS}{\mathcal{S}}

\newcommand{\herA}{A^{\textup{H}}}

\newcommand{\scp}[2]{\left\langle{#1}, {#2}\right\rangle}
\newcommand{\scpB}[2]{\left\langle{#1}, {#2}\right\rangle_B}

\DeclarePairedDelimiter{\abs}{\vert}{\vert}

\DeclarePairedDelimiter{\norm}{\|}{\|}
\DeclarePairedDelimiter{\normB}{\|}{\|_{B}}

\DeclareMathOperator{\sign}{sign}
\DeclareMathOperator{\diag}{diag}

\let\dim\relax
\DeclareMathOperator{\dim}{dim}
\DeclareMathOperator*{\argmax}{argmax}

\DeclareMathOperator*{\dist}{dist}

\DeclareMathOperator*{\grad}{grad}

\newcommand{\tT}{\mathrm{T}}
\newcommand{\pd}{\mathrm{Sym}_{\succ 0}^d}

\newcommand{\RE}{\operatorname{Re}}
\newcommand{\IM}{\operatorname{Im}}
\newcommand{\dx}{\mathrm{d}}

\newtheorem{theorem}{Theorem}[section]

\newtheorem{lemma}[theorem]{Lemma}
\newtheorem{remark}{Remark}[section]

\begin{document}

\title{Stochastic Zeroth-Order Method for Computing  Generalized Rayleigh Quotients}

\author{Jonas Bresch\thanks{Technische Universität Berlin, Straße des 17. Juni 136, Berlin, 10587, Germany} \\
{\footnotesize\href{mailto:bresch@math.tu-berlin.de}{bresch@math.tu-berlin.de}}
\and Oleh Melnyk\footnotemark[1] \\
{\footnotesize\href{mailto:melnyk@math.tu-berlin.de}{melnyk@math.tu-berlin.de}} 
\and Martin Schoen\footnotemark[1]
\and Gabriele Steidl\footnotemark[1] \\
{\footnotesize\href{mailto:steidl@math.tu-berlin.de}{steidl@math.tu-berlin.de}}}

\maketitle

\begin{abstract}
The maximization of the (generalized) Rayleigh quotient is a central problem in numerical linear algebra. Conventional algorithms for its computation typically rely on matrix–adjoint products, making them sensitive to errors arising from adjoint mismatches. To address this issue, we introduce a stochastic zeroth-order Riemannian algorithm that maximizes the generalized Rayleigh quotient without requiring adjoint or matrix inverse computations. We provide theoretical convergence guarantees showing that the iterates converge to the set of global maximizers of the (generalized) Rayleigh quotient and the norm of the Riemannian gradient vanishes at a sublinear rate with probability one.
Our theoretical results are supported by numerical experiments, which demonstrate the excellent performance of the proposed method compared to state-of-the-art algorithms.
\end{abstract}

\textbf{Keywords.}
generalized Rayleigh quotient $\cdot$
spectral norm $\cdot$
stochastic optimization $\cdot$
zeroth-order optimization $\cdot$
Riemannian optimization

\textbf{MSC.} 58C40 
$\cdot$ 65F15  
$\cdot$ 65F35 
$\cdot$ 15A60 
$\cdot$ 68W20

\section{Introduction}\label{sec:intro}
In this paper, 
we are interested in the maximization of the generalized Rayleigh quotient  
\begin{equation}\label{eq: intro quotient}
    \mathcal R(A,B) = \max_{v \in \RR^d\setminus\{0\}} \frac{\scp{v}{Av}}{\scp{v}{Bv}}
\end{equation}
without explicitly using the inverse of the positive definite $B \in \mathbb R^{d\times d}$ 
or the transpose of $A\in \mathbb R^{d\times d}$.
The maximization of \eqref{eq: intro quotient} is a fundamental problem 
in various applications such as spectral equivalence 
of the matrices $A$ and $B$ \cite{axelsson2001finite}, 
generalized singular value \cite{golub2013}, 
and tensor \cite{KoMa14} decompositions.
For the identity matrix $B$, 
the Rayleigh quotient is also known as the numerical abscissa 
and is used for the stability analysis of nonsymmetric matrices 
in partial differential equations \cite{farrell1996generalizedstabilitytheorypartIInonautonomousoperators, trefethen2005spectrapseudospectra, benzi2021someusenumericalanalysis}.

The maximization of \eqref{eq: intro quotient} 
can be performed by a number of methods. 
Classical iterative schemes for real eigenvalue problems 
use Rayleigh quotient iterations \cite{vonMises1929power, Muntz1913a, Muntz1913b, parlett1974raylightquotientiteration} 
and their block \cite{knyazev2007block}, Krylov-style \cite{parlett1998symmetriceigenvalueproblem},
or matrix-free \cite{knyazev2001matrixfreekrylovpcgmethod} variants.
A more robust class of approaches 
is based on the min--max characterizations 
for generalized eigenvalue problems \cite{schanze2023robustraylight, zhaojun2018robustraylight, nishioka2025minmaxgeneralizedeigenvalue}. 
Alternatively, 
maximization of \eqref{eq: intro quotient} 
can be performed using Riemannian optimization techniques \cite{zhang2016riemannian,alimisis2021distributed, alimisis2024geodesic}.
Another class of algorithms relies on constructing rank-one perturbations of $A$ 
leading to a stable approximation of the largest eigenvalue
\cite{guglielmi2011fastaglorithmapproxpseudoabscissa}.
These can be prohibitively expensive for large-scale matrices, 
and a more scalable algorithmic approach based on subspace methods 
was proposed in \cite{kressner2014subspacepseudospectralabscissa,ding2017computingrealpseudoabscissa}. 
Sketching methods \cite{halko2011randomapproximatematrixdecomposition, li2014sketching, tropp2020computationalframework} 
solve the problem using random dimensional reduction techniques. 
All the above methods rely on matrix–adjoint product 
and/or require access to the inverse of $B$.

The motivation for developing inverse- and adjoint-free methods 
stems from two main considerations.
First, 
computing $B^{-1}$ or performing matrix–vector products involving it,
is computationally expensive and susceptible to numerical inaccuracies. 
Second, 
in imaging applications 
such as computed tomography \cite{buzug2008ct, xie2015effective, zhang2016unmatchedprojback, peterson2017monte}, 
the transpose $A^\tT$ is often replaced by an approximate,
but computationally tractable operator. 
This substitution introduces what is known as adjoint mismatch, 
which can lead to significant reconstruction errors.

When $A^\tT$ and $B^{-1}$ are unavailable, 
zeroth-order optimization methods can be employed, 
as they rely solely on evaluation of products with $A$ and $B$. 
These methods approximate the gradient using finite-difference schemes \cite{chan1998transpose,balasubramanian2022zeroth,li2023stochastic}. 
In the context of Rayleigh quotient optimization,
the bundled gradient method \cite{burke2002optimizingmatrixstability} 
or Oja's algorithm and its versions \cite{oja1982neuronmodelpca} 
were employed, which require very limited storage 
and involve computationally simple updates. 
Another promising approach is consensus-based optimization \cite{riedl2024consensusbasedoptimization,fornasier2025regularitypositivitysolutionsconsensusbased,fornasier2021consensusbasedhypersurfaces,fornasier2021consensusbasedsphere},
where the search space is explored by a system of interacting particles 
governed by coupled stochastic differential equations
that balance random exploration with attraction 
toward the current best estimate of the optimum.

In this paper, 
we propose a new simple and efficient algorithm 
for solving \eqref{eq: intro quotient} 
that combines ideas from stochastic zeroth-order Riemannian optimization \cite{li2023stochastic} 
and slicing methods \cite{quellmalz2023slicing,quellmalz2024slicing}. 
Namely, 
we iteratively solve the maximization problem 
in a randomly sampled one-dimensional subspace. 
This yields a provably convergent algorithms 
that neither involve the computation of $B^{-1}$ nor $A^\tT$, 
and each iteration requires only a limited number of matrix-vector products. 
Numerical experiments show the outstanding performance 
of our method compared to other techniques in the literature. 
Our algorithms can be extended for computing the maximal real generalized Rayleigh quotient 
for complex matrices $A,B \in \mathbb C^{d \times d}$.

\paragraph{Outline of the Paper}
In Section~\ref{sec:preliminaries}, 
we start with preliminaries on generalized Rayleigh quotients 
and first- and zeroth-order  Riemannian optimization. 
Our algorithm is presented in Section~\ref{sec:gRq}.
We study its termination behavior in dependence on the dimension of the
eigenspace belonging to $\mathcal R(A,B)$
in Section~\ref{sec:termination}. 
If this dimension is smaller than $d-1$,  
we derive convergence to a global maxima in Section~\ref{sec:conv} 
and establish convergence rates in Section~\ref{sec:convergence_rates}. 
Section~\ref{sec: zeroth-order link} gives an interesting reinterpretation of our algorithms 
as Riemannian zeroth-order method. 
The theoretical findings are substantiated by numerical examples, 
including comparisons with other algorithms in Section~\ref{sec:num}.
Appendix~\ref{sec:proofs} contains proofs of technical lemmas. 
Remarkably, 
Appendix~\ref{sec: complex} extends our algorithm to the complex setting $A,B \in \mathbb C^{d \times d}$.

This paper extends the preprint \cite{bresch2024matrixfreestochasticcalculationoperator} 
presented by one of the authors at the ILAS Conference 2025
to generalized Rayleigh quotients and Riemannian geometry, 
and incorporates additional clarifications and substantial improvements.

\section{Preliminaries}\label{sec:preliminaries}
We start by recalling generalized Rayleigh quotients
and methods for its computation from the Riemannian optimization point of view.

\subsection{Generalized Rayleigh Quotients}
Throughout the paper, let
$\norm{\cdot}$ be the Euclidean norm induced by the inner product $\scp{\cdot}{\cdot}$
and $I_d \in \RR^{d\times d}$ the identity matrix.
For a matrix $A \in \RR^{d\times d}$, we denote by
\[
    \herA \coloneqq \tfrac{1}{2}(A +  A^\tT) 
\]
its symmetric part. By the symmetry of the inner product, we have 
\begin{equation}\label{sym}
    \langle v,Av \rangle = 
    \tfrac12 (\langle v,Av \rangle + \langle Av,v \rangle) 
    =
    \tfrac12 (\langle v,Av \rangle + \langle v,A^\tT v \rangle) 
    = \langle v, \herA v \rangle.
\end{equation}
Let $\pd \subset \RR^{d\times d}$ be the set of symmetric, positive definite matrices.
Then, for $A \in \RR^{d\times d}$ and $B \in \pd$, 
we define the \emph{generalized Rayleigh quotient} of $A$ and $B$ with respect to 
$v \in  \RR^d\setminus\{0\}$ by
\begin{equation}\label{eq:ray}
    r(A, B, v) 
        \coloneqq \frac{\scp{v}{Av}}{\scp{v}{Bv}}
    =\frac{\scp{v}{\herA v}}{\scp{v}{Bv}}.
\end{equation}
We are interested in computing its maximum 
\begin{equation}    \label{eq:gRq}
    \mathcal R(A, B) 
    \coloneqq \max_{v \in \RR^d\setminus\{0\}} r(A, B, v) = \max_{v \in \sphere_B^{d-1}} \; \scp{v}{A v},
\end{equation}
where  $\sphere_B^{d-1}$ denotes the unit sphere with respect to $B$ given by
\begin{equation}
    \sphere_B^{d-1} \coloneqq \{v \in \RR^d : \normB{v} = 1\}
\end{equation}
and $\normB{v}^2 \coloneqq \scpB{v}{v} \coloneqq \scp{v}{Bv}$.
For $B = I_d$, we just set $\mathbb S^{d-1} \coloneqq \sphere_{I_d}^{d-1}$.
Figure \ref{fig:R2x2_problems} visualizes the
Rayleigh quotient and, in particular, the constrained part of interest on $\sphere_B^{d-1}$.

\begin{figure}[t!]
\begin{center}
    \includegraphics[width=0.35\linewidth, clip=true, trim=0pt 25pt 40pt 50pt]{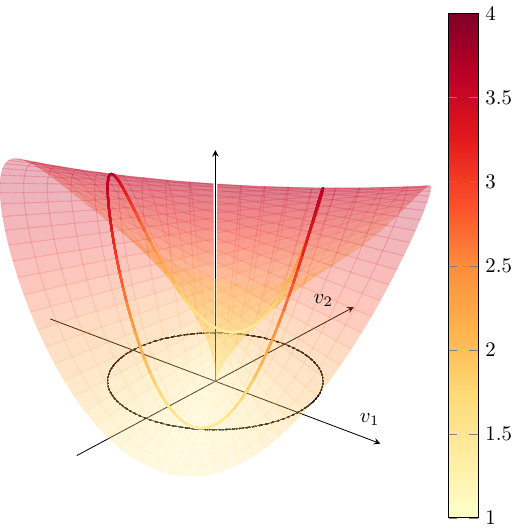}
    \hspace{0.3cm}
    \includegraphics[width=0.35\linewidth, clip=true, trim=0pt 25pt 40pt 50pt]{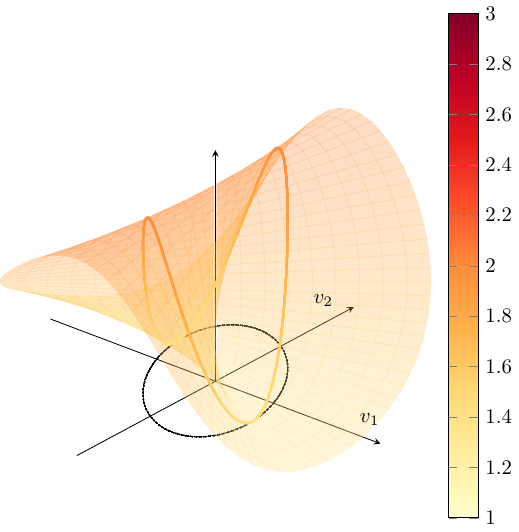} 
    \hspace{0.05cm}
    \raisebox{10pt}{\includegraphics[width=0.04\linewidth, clip=true, trim=210pt 0pt 0pt 0pt]{pdf_plots/R2x2_problem_bar.pdf}}
    \end{center}
    \caption{%
    Visualization of  the generalized Rayleigh quotient $r(A,B,v)$ 
    for $v = (v_1,v_2)^\tT \in \RR^2\setminus\{0\}$ 
    and 
    $A = \left(\begin{smallmatrix}
            3 & 1 \\ 1 & 2        
        \end{smallmatrix}\right)$.
    Its values on $ \sphere_B^1$ are highlighted by the solid line.
    \emph{Left}: $B= I_2$, \emph{Right}: $B= \left(\begin{smallmatrix}
        2 & 0.5 \\ 0.5 & 1        
    \end{smallmatrix}\right)$.
        }
    \label{fig:R2x2_problems}
\end{figure}

A value $\lambda \in \mathbb R$ is called a 
\emph{generalized eigenvalue} of $(\herA,B)$, it there exists
$v,u \in \RR^d\setminus\{0\}$ fulfilling
    \begin{equation}\label{g_eigenvalue}
    \herA v = \lambda B v
    \quad \text{or, equivalently,} \quad
    B^{-1} \herA v = \lambda v.
\end{equation}
The eigenvalues $\lambda_k(B^{-1}\herA)$ 
and eigenvectors $v^k$ of $B^{-1} \herA$ are also identified
with eigenvalues \smash{$\lambda_k(B^{-\frac12} \herA  B^{-\frac12})$}
and eigenvectors $u^k$ of a symmetric matrix \smash{$B^{-\frac12} \herA  B^{-\frac12}$}
through \smash{$\lambda_k(B^{-1}\herA) = \lambda_k(B^{-\frac12} \herA  B^{-\frac12})$} and $B^{\frac{1}{2}}u^k = v^k$.
We denote the corresponding eigenspaces by
\begin{equation*}
    G_\lambda \coloneqq \{v \in \RR^d\setminus \{0\}: 
    \herA v = \lambda B v \}.
\end{equation*}
Since eigenvectors of $B^{-\frac12} \herA  B^{-\frac12}$ span $\RR^d$, 
same holds true for eigenvectors of $B^{-1} \herA$ 
by $B$ being a full rank matrix. 
Yet, 
eigenspaces $G_\lambda$ are no longer orthogonal to each other. 
The value $\mathcal R(A,B)$ is the largest generalized eigenvalue of $(\herA,B)$
and we set 
\begin{equation*}
    G_{\max} \coloneqq G_{\mathcal R(A,B) } 
    \quad \text{and} \quad
    G \coloneqq \bigcup_{\lambda \text{ gen. eigen-} \atop \text{value of }(\herA,B)}  G_\lambda.
\end{equation*}
\begin{remark}[Numerical Abscissa]
The numerical range of $A \in \RR^{d \times d}$ is defined by
\begin{equation*}
    W(A) \coloneqq \Biggl\{
    \frac{\scp{z}{A z}}{\|z\|^2} : z \in \CC^d\setminus\{0\}
    \Biggr\}
    \qquad \text{and} \qquad 
    \omega(A) \coloneqq \max_{\lambda \in W(A)} \RE(\lambda)
\end{equation*}
is called \emph{numerical abscissa} of $A$ \cite[Eq.~17.22]{trefethen2005spectrapseudospectra}. 
This is equal to the largest eigenvalue 
of  $\herA$ by \cite[Eq.~2.2]{mitchell2023numradius}.
\end{remark}
In this paper, we are interested in the computation of $\mathcal R(A,B)$ without explicitly using the inverse of $B$ or the transpose of $A$. We will propose a new zeroth-order method on $\sphere^{d-1}_B$, which we prepare in the next section.

\subsection{Riemannian First- and Zeroth-Order Optimization on $\sphere^{d-1}_B$}
We review the main concepts of the Riemannian gradient ascent method 
on $\sphere^{d-1}_B$ based on \cite[§~3]{boumal2023introduction} 
and a stochastic zeroth-order method,
which we will use in the numerical part for comparisons.

The $(d-1)$-dimensional embedded submanifold $\sphere_B^{d-1} \subset \RR^d$
has the tangential space $T_v$ at $v \in \sphere^{d-1}_B$,
\begin{equation*}
    T_v = \{x \in \RR^d : \scp{x}{B v} = 0\},
\end{equation*}
and its orthogonal complement reads as 
\begin{equation*}
    T_v^\perp = \{ \lambda B v : \lambda \in \RR \}.
\end{equation*}
In particular, 
we see by \eqref{g_eigenvalue},
if $v \in \sphere^{d-1}_B \cap G$, 
then 
\begin{equation} \label{helper}
    \herA v \in T_v^\perp \quad \text{and} \quad
    (\text{span} \, \herA v)^\perp = T_v.
\end{equation}
The orthogonal projection $P_{v}: \RR^d \to T_v$ is given, for $y \in \RR^d$, by
\begin{equation*}
    P_v y = 
    \Bigl(I_d - \tfrac{Bv}{\norm{Bv}}\tfrac{(Bv)^\tT}{\norm{Bv}}\Bigr)y = 
    y - \langle  y, \tfrac{Bv}{\norm{Bv}} \rangle  \tfrac{Bv}{\norm{Bv}} .
\end{equation*}
To map points from $T_v$ back to $\sphere_B^{d-1}$, we will use the retraction $R_v : T_v \to \sphere_B^{d-1}$ defined by 
\begin{equation}
    \label{eq: retraction}
    R_v(x) \coloneqq \frac{v + x}{\normB{v + x}}.
\end{equation}
For a smooth function $f : \sphere_B^{d-1} \to \RR$, the Riemannian gradient is given by
\begin{equation}    \label{eq: riemannian gradient}
    \grad f(v) 
      = P_{v} \nabla \bar f(v) \in T_v
\end{equation}
where $\bar f : \RR^d \to \RR$ with $\bar f|_{\sphere_B^{d-1}} = f$ 
is a smooth extension of $f$ to $\RR^d$.
For optimizing \eqref{eq:gRq}, we consider
\begin{equation} \label{problem_riemann}
    f(v) \coloneqq \scp{v}{\herA v}, \quad \text{and} \quad
    \bar f(v) = \scp{v}{\herA v}, \quad v \in \RR^d.
\end{equation}
Its Riemannian gradient  is 
\begin{equation}\label{eq: riemannian gradient special}
    \grad f(v) = P_v \nabla \bar f(v)
    = 2
    \bigl(
    \herA v - \tfrac{ (B v)^\tT(\herA v)}{\norm{Bv}^2} 
    B v
    \bigr)
\end{equation}
The critical points $\{v \in \sphere_B^{d-1}: \grad f(v) =0\}$ 
admit $\herA v \in T_v^\perp$.

To solve \eqref{problem_riemann},
we can apply the \emph{Riemannian gradient ascent method}: 
starting in an arbitrary
$v^0 \in \sphere_B^{d-1}$, 
we compute for step size $\tau_k > 0$, $k \in \NN$ the updates 
\begin{equation}
    \label{eq: Riemmanian gradient ascent}
    v^{k+1} = R_{v^k}(\tau_k \, \grad f(v^k) ) 
    = \frac{v^{k} + \tau_k \, \grad f(v^k)}{\normB{v^{k} + \tau_k \, \grad f(v^k)}}.
\end{equation}
The convergence of the Riemannian gradient ascent method is ensured by the following theorem.

\begin{theorem}\label{prop1}
    Let $f$ be defined by \eqref{problem_riemann} and
    $L \ge 2 \norm{\herA}(1 + \kappa(B))$, 
    where $\kappa(B) \coloneq \|B\| \|B^{-1}\|$ is the condition number of $B$.
    Then the sequence $(v^k)_{k=0}^{\infty}$ generated by \eqref{eq: Riemmanian gradient ascent} 
    with $\tau_k = 1/L$ fulfills 
    $\grad f(v^k) \to 0$ as $k \to \infty$ at a sublinear rate, meaning that
\begin{equation*}
    \label{eq: conv sublinear}
    \min_{k=0,\ldots,n} \norm{\grad f(v^k)}^2 
    \le \frac{2 L}{n+1} \big(\mathcal R(A,B) - f(v^0) \big).
\end{equation*}    
\end{theorem}

\begin{proof}
    The statement follows from \cite[Cor.~4.8]{boumal2023introduction} 
    applied for minimization of $-f$. 
    We can apply the cited statement if $-f$ is bounded from below and fulfills
    \begin{equation}
        \label{eq: retraction smoothness}
        | f(R_v(x)) - f(v) - \scp{\grad f(v)}{x}| \le \frac{L}{2} \norm{x}^2
    \end{equation}
    for all $x \in T_v$ and $v \in \sphere_{B}^{d-1}$.
    Indeed, our function $-f$ is bounded from below by $-\mathcal R(A,B)$.
    Further,
    \[
        \scp{\grad f(v)}{x} = 2\scp{P_v \herA v}{x} = 2\scp{\herA v}{P_v x} = 2\scp{\herA v}{x}
    \]
    and
    \begin{align*}
        f(v) + \scp{\grad f(v)}{x} 
        & = \scp{v}{\herA v} + 2\scp{\herA v}{x} \\
        & = \scp{v + x}{\herA (v +x)} - \scp{x}{\herA x} \\
        & = \normB{v + x}^2 \scp{R_v(x)}{\herA R_v(x)} - \scp{x}{\herA x}\\
        &= (1 + \normB{x}^2)\scp{R_v(x)}{\herA R_v(x)} - \scp{x}{\herA x}
    \end{align*}
    yields with 
    \begin{equation} \label{eq: norm bounds}
        \|B^{-1}\|^{-1} \|y\|^2\ \le \|y\|^2_B \le \|B\| \|y\|^2  
        \quad \text{for all} 
        \quad y \in \RR^d  
    \end{equation} 
    that
    \begin{align*}
        | f(R_v(x)) - f(v) - \scp{\grad f(v)}{x} | 
         &
         = | \scp{R_v(x)}{\herA R_v(x)}  \normB{x}^2
        + \scp{x}{\herA x} | \\
        &  
        \le \norm{\herA} ( \norm{R_v(x)}^2  \normB{x}^{2} 
        + \norm{x}^2)\\
        &\le \norm{\herA} ( \norm{R_v(x)}^2  \|B\| + 1) \norm{x}^2\\
        &\le \norm{\herA} (  \|B\| \cdot \|B^{-1}\| + 1) \norm{x}^2,
    \end{align*}
    and by \eqref{eq: retraction smoothness} with $L \geq 2 \|\herA\| (1 + \kappa(B))$.
\end{proof}

In the special case $B = I_d$ and $\herA \in \pd$,
the convergence result can be strengthened to convergence to the set of global maximizers 
based on \cite{alimisis2021distributed}.

Computing $\grad f$ requires evaluating products with $A^\tT$,
which we want to avoid. This can be done by using, 
e.g. the \emph{stochastic zeroth-order optimization} 
proposed for general embedded manifolds in \cite{li2023stochastic}: 
starting in a arbitrary $v^0 \in \sphere_B^{d-1}$, 
the methods computes for step size $\tau_k > 0$, $k \in \NN$ the updates  
\begin{equation}
    \label{eq: zero order ascent}
    v^{k+1} 
    \coloneqq R_{v^k}(\tau_k \, \widehat{\grad}_m f(v^k) ) 
    = \frac{v^{k} + \tau_k\,  \widehat{\grad}_m f(v^k)}{\normB{v^{k} + \tau_k \, \widehat{\grad}_m f(v^k)}},
\end{equation}
where $\widehat{\grad}_m$ is \emph{$m$-sample approximation of the Riemannian gradient} 
\begin{equation}
    \label{eq: m-sample gradient}
    \widehat{\grad}_m f(v)
    \coloneq \frac{1}{m}\sum_{i =1}^m \frac{f(R_v(\mu P_v x_i)) - f(v)}{\mu} P_v x_i, \quad x_i \sim \cN(0, I_d)
\end{equation}
with scaling parameter $\mu > 0$. 
Note that sampling from the Gaussian distribution 
can be replaced by any rotation-invariant distribution. 
Further, the zeroth-order iteration \eqref{eq: zero order ascent}
can be seen as an instance of a larger class of inexact gradient methods, 
see, e.g., \cite{zhou2025inexact}. 
As stated in the next theorem,
the general convergence result from \cite{li2023stochastic} can be applied
to our special minimization problem \eqref{eq:gRq}.
We postpone the proof to Appendix~\ref{sec:proofs}.

\begin{theorem}
\label{th: zero-order convergence}
    Let $f$ be defined by \eqref{problem_riemann} and
    $L \ge 2 \norm{\herA}(1 + \kappa(B))$. 
    Then the sequence $(v^k)_{k=0}^{\infty}$ generated by \eqref{eq: zero order ascent} 
    with $\tau_k = 1/\big(2(d+4) L \big)$ and scaling parameters $\mu_k$ 
    satisfying $\sum_{k \in \NN} \mu_k^2 < \infty$ 
    fulfills $\grad f(v^k) \to 0$ a.s.\ as $k\to \infty$ 
    and there exists a constant $C > 0$ depending on $L$ and $d$ such that  
    \[
        \min_{k=0,\ldots,n} \EE [\norm{\grad f(v^k)}^2] 
        \le \frac{8(d+4) L}{n+1}[\mathcal R(A,B) - f(v^0) + C \sum_{k=0}^{\infty} \mu_k^2].
    \]
\end{theorem}

\section{New Stochastic Zeroth-Order Algorithm} \label{sec:gRq}
In this section, we propose a simpler 
and more powerful stochastic zeroth-order algorithms for maximizing the generalized Rayleigh quotient.
To make the convergence analysis better accessible, 
we start with a one-sample method given in Algorithm \ref{alg:gRq}, i.e., $m=1$. Then the algorithm is generalized
to more samples.

\begin{algorithm}[h!]
  \caption{One-sample zeroth-order method}\label{alg:gRq}
  \begin{algorithmic}[1]
    \State Initialize $v^{0} \coloneqq  \tilde v/\|\tilde v\|_B, \quad \tilde v \sim \mathcal N(0,I_d)$
    \For{$k=0,1,2,\dots$}
      \State Sample  $$x^k \coloneqq P_{v^k} \tilde x /\|P_{v^k} \tilde x\|, 
      \quad
      \tilde x \sim \mathcal N(0,I_d)$$ 
      \State If $\langle x^k, \herA v^k \rangle = 0$ then \textbf{stop.}
      \State Otherwise, calculate step size 
      $$
      \tau_{k} \coloneqq \argmax\limits_{\tau \in \RR} r(A,B, v^k + \tau x^k)
      $$
     \State Update  $$v^{k+1} \coloneqq \frac{v^k + \tau_k x^k}{\normB{v^k + \tau_k x^k}}$$ 
    \EndFor
  \end{algorithmic}
\end{algorithm}

Sampling from a Gaussian, 
we clearly have that $\tilde v = 0$ 
as well as $P_{v^k} \tilde x = 0$ occurs with probability zero. 
By the following theorem, 
the step sizes $\tau_k$ in each iteration $k$ in Algorithm~\ref{alg:gRq} 
can be computed analytically. 
To this end, we note that by construction
$v^k \in \sphere^{d-1}_B$ and $x^k \in T_{v^k} \cap \mathbb S^{d-1}$.
To simplify the notation, 
we drop the index and superscript $k$ for counting the iteration.

\begin{theorem}\label{prop:stepsize}
For $v \in \mathbb S^{d-1}_B$ and $x \in T_v \setminus \{0\}$, let 
\begin{align}
    a &\coloneqq \scp{v}{A v},    
    \quad b \coloneqq \scp{x}{A v} + \scp{v}{A x} = 2\scp{x}{\herA v},\\
     c &\coloneqq \scp{x}{A x},   
    \quad d \coloneqq \scp{x}{B x}. 
\end{align}
If $b \not = 0$, then
\begin{equation} \label{maxi}
\tau^* \coloneqq \argmax_{\tau \in \RR} r(A,B, v + \tau x) 
\end{equation}
is given by
\begin{equation}\label{eq:tau}
\tau^* = \sign(b)\left(\frac{c - a d}{\abs{b} d} + \sqrt{\tfrac{(c - a d)^2}{(b d)^2} + \tfrac{1}{d}}\right) \in \RR \backslash \{0\}.
\end{equation}
\end{theorem}

We emphasize that the computation of $b$ requires just access to $A$, but not to $A^\tT$.

\begin{proof}
By definition \eqref{eq:ray} and since $x \perp Bv$, we obtain
\begin{equation}\label{eq:func_gRq_abcd}
   g(\tau) \coloneqq  r(A,B,x+\tau v) 
    =  
    \frac{\langle x+\tau v, A(x+\tau v) \rangle}{\langle x+\tau v, B (x+\tau v) \rangle}  
     = \frac{a + \tau b + \tau^2 c}{1+\tau^2 d}
\end{equation}
so that
\begin{equation}\label{eq: derivative r_k}
   g'(\tau)
    = \frac{b + 2 \tau (c - a d) - \tau^2 b d}{(1 + \tau^2 d)^2} = 0.
\end{equation}
Since $d >0$ and $b \not = 0$, this equation has the solutions 
\begin{equation}
\tau_\pm =
        \frac{c - a d}{b d}
        \pm \sqrt{\tfrac{(c - a d)^2}{(b d)^2} + \tfrac{1}{d}}.
 \end{equation}   
 Straightforward computation shows 
 that for $b>0$ the maximum is attained 
 at $\tau_+$ and for $b<0$ at $\tau_-$. 
 Combining both cases,
 we conclude that $\tau^*$ reads as in \eqref{eq:tau}.
\end{proof}

\begin{remark}[Sub-Rayleigh quotient problem]
    \label{rem:Sub-Rayleigh quotient problem}
    The computation of the optimal step size from Theorem~\ref{prop:stepsize} 
    is equivalent to the Rayleight quotient problem of 
    \begin{equation*}
        A_2 \coloneqq \left[\begin{smallmatrix}
            a & \nicefrac{b}{2} \\ \nicefrac{b}{2} & c
        \end{smallmatrix}\right]
        = \left[\begin{smallmatrix}
            \scp{v}{Av} & \scp{x}{Av} \\ \scp{x}{Av} & \scp{x}{Ax}
        \end{smallmatrix}\right]
        = \left[\begin{smallmatrix}
            v & x
        \end{smallmatrix}\right]^{\tT}
        A \left[\begin{smallmatrix}
            v & x
        \end{smallmatrix}\right]
   \end{equation*}
   and
   \begin{equation*}
        B_2 \coloneqq \diag(1, d)
        = \left[\begin{smallmatrix}
            \scp{v}{Bv} & \scp{x}{Bv} \\ \scp{x}{Bv} & \scp{x}{Bx}
        \end{smallmatrix}\right]
        = \left[\begin{smallmatrix}
            v & x
        \end{smallmatrix}\right]^{\tT}
        B \left[\begin{smallmatrix}
            v & x
        \end{smallmatrix}\right].
    \end{equation*}
    Hence, it holds
    \begin{equation*}
        \tau^* = \argmax_{\tau \in \RR} \frac{\scp{\left(\begin{smallmatrix}
            1 \\ \tau 
        \end{smallmatrix}\right)}{ A_2 \left(\begin{smallmatrix}
            1 \\ \tau 
        \end{smallmatrix}\right)}}{\scp{\left(\begin{smallmatrix}
            1 \\ \tau 
        \end{smallmatrix}\right)}{ B_2 \left(\begin{smallmatrix}
            1 \\ \tau 
        \end{smallmatrix}\right)}}
        = \argmax_{\tau \in \RR}
        \frac{\scp{v + \tau x}{A (v + \tau x)}}{\scp{v + \tau x}{B (v + \tau x)}},
    \end{equation*}
    If $B = I_d$, then $B_2 = I_2$ and the computation of the step size 
    is the Rayleigh-Ritz method \cite{Rayleigh2011,Ritz1909}.
\end{remark}

A multi-sample version of Algorithm~\ref{alg:gRq} is provided by the following Algorithm~\ref{alg:gRq m-sample}.  For a motivation, see Remark \ref{discussion}.

\begin{algorithm}[h]
  \caption{$m$-sample zeroth-order method}\label{alg:gRq m-sample}
  \begin{algorithmic}[1]
    \State Initialize Initialize $v^{0} \coloneqq  \tilde v/\|\tilde v\|_B, \quad \tilde v \sim \mathcal N(0,I_d)$
    \For{$k=0,1,2,\dots$} 
      \State Sample 
       $$x^{k,i} \coloneqq P_{v^{k}} \tilde x^i /\|P_{v^k} \tilde x^i\|, 
      \quad
      \tilde x^i \sim \mathcal N(0,I_d), \quad i=1,\ldots,m$$  
      \State Construct
      \begin{align*}
          \bar x^{k} &= \tfrac{1}{m} \sum_{i=1}^{m} b_{k,i} \, x^{k,i}, \quad b_{k,i} \coloneqq 2 \langle x^{k,i}, \herA v^k \rangle \\
          x^{k} &= \bar x^{k} / \norm{\bar x^{k}}
      \end{align*}
      \State If $\langle x^k, \herA v^k \rangle = 0$ then \textbf{stop.}
      \State Otherwise, calculate step size 
      $$
      \tau_{k} \coloneqq \argmax\limits_{\tau \in \RR} r(A,B, v^k + \tau x^k)
      $$
      \State Update  $$v^{k+1} \coloneqq \frac{v^k + \tau_k x^k}{\normB{v^k + \tau_k x^k}}$$ 
    \EndFor
  \end{algorithmic}
\end{algorithm}

In the following sections, we provide convergence results 
for Algorithm~\ref{alg:gRq}. 
We give only an informal justification that all these results 
can be extended to Algorithm~\ref{alg:gRq m-sample}:  
we show in the next section for Algorithm~\ref{alg:gRq} 
that for $\dim (G_{\max}) < d-1$, 
the case $b_k \coloneqq 2 \scp{x^k}{\herA v^k} = 0$ appears with probability zero 
with respect to the uniform measure on a sphere. 
Although the conditional distribution of $x^k$ given $v^k$ 
constructed in Algorithm~\ref{alg:gRq m-sample} is no longer uniform, 
it is still absolutely continuous, for which the results can be obtained similarly.

\section{Termination of the Algorithm} \label{sec:termination}
In this section, we deal with the termination behavior of Algorithm~\ref{alg:gRq}.
Our main result in Theorem~\ref{l: initialization main}  
states that for $\dim(G_{\max}) < d-1$, 
the values $b_k$, $k \in \mathbb N$ will ``in general'' not vanish, 
so that the algorithm does not terminate
until an adjusted stopping criterion is reached.

The following remark sets up our stochastic setting.

\begin{remark}[Distribution of $x^k$ and $v^k$]    \label{rem:dist_x_k}
We consider a probability space  $(\Omega, \mathcal A, \mathbb P)$
and random variables $V: \Omega \to \sphere^{d-1}_B$ with law
\smash{$\mathbb P_V = V_\sharp \mathbb P = \mathbb P \circ V^{-1}$}.
Further, given a fixed $v \in \mathbb R^d$, we  deal with random variables
$X:  \Omega \to T_v \cap \mathbb S^{d-1}\simeq  \mathbb S^{d-2}$ 
which are uniformly distributed \smash{$\mathbb P_X = \mathcal U(T_v \cap \mathbb S^{d-1})$}
by the following Lemma~\ref{l: direction main}.
Therefore, we deal with surface measures $\sigma_{\sphere^{d-2}}$.
In the realm of the algorithm, these random variables become conditional ones with laws \smash{$\mathbb P_{X^k|V^k = v^k} = \mathcal U(T_{v^k} \cap \mathbb S^{d-1})$}.
Then the updates $v^{k+1}$ are samples from a conditional random variable with law
\smash{$\mathbb P_{V^{k+1}|V^k = v^k}$} depending just on the previous step.
In other words, 
\smash{$\mathbb P_{V^{k+1}|V^k}$}
are Markov kernels, meaning that 
\smash{$\mathbb P_{V^{k+1}|V^k = v^k}$} is a measure
for any $v^k \in \sphere^{d-1}_B$ and \smash{$\mathbb P_{V^{k+1}|V^k = \cdot} (A)$} is a measurable function for any Borel set $A \subset \sphere^{d-1}_B$.
Then our algorithm produces samples from a Markov chain
$(V^k)_{k \in \NN}$.

Finally, note that we use $x \sim \mathcal N(0,I_d)$, if $x$ is sampled from the standard normal distribution and $X \sim \mathcal N(0,I_d)$
to say that the random variable $X$ is standard normally distributed.
\end{remark}

As announced in the remark, 
we will need the following lemma, 
the proof of which is given in the Appendix~\ref{sec:proofs}. Note that the proof
specifies the isomorphism $T_{v} \cap \sphere^{d-1} \simeq \sphere^{d-2}$.

\begin{lemma}\label{l: direction main}
For a fixed $v \in  \sphere_B^{d-1}$  and $\tilde X \sim \cN(0, I_d)$, the random variable
    $X \coloneq P_v \tilde X/\|P_v \tilde X\|$
    is uniformly distributed 
    on $T_{v} \cap \sphere^{d-1} \simeq \sphere^{d-2}$.
    Moreover, it holds
    \begin{align*}
        \EE_{x \sim X}[x x^\tT] 
       = \frac{1}{d-1} P_{v} = \frac{1}{d-1} \Bigl(I_d - \tfrac{Bv}{\norm{Bv}}\tfrac{(Bv)^\tT}{\norm{Bv}}\Bigr) .
    \end{align*}
\end{lemma}

Further, the following fact is required, see Appendix~\ref{sec:proofs} for the proof.

\begin{lemma}\label{l: direction main_iii}    
     Let $M$ be an affine subspace in $\RR^d$ of dimension $r$,
    and     $$\varphi : M \setminus\{0\} \to \sphere^{d-1}, x \mapsto x/\norm{x}.$$
    \begin{itemize}
        \item[{\textrm i)}]
        If $0 \in M$ and $r < d$, then $\varphi(M)$ is of measure zero  with respect to  the surface measure $\sigma_{\sphere^{d-1}}$.
      \item[{\textrm ii)}]  
    If $0 \not \in M$ and $r < d-1$, then
        $\varphi(M)$ is of measure zero 
        with respect to $\sigma_{\sphere^{d-1}}$.
        \end{itemize}
\end{lemma}

We have to distinguish the three cases $\dim(G_{\max})\in \{d-1, d\}$
and $\dim(G_{\max}) < d-1$. The first case is handled in the following remark.

\begin{remark}[Special Cases of Dimension]\label{rem:1} \hfill
\begin{itemize}
\item[{\textrm i)}]    
If $\dim(G_{\max}) = d$, then every vector of $\RR^d$ is a generalized eigenvector of $(\herA, B)$ belonging to the eigenvalue $\mathcal R(A,B)$ and thus is a maximizer
    of $r(A,B,\cdot)$. In particular, we have $v^0 \in G_{\max}$. 
    But then $\herA v^0 =  \mathcal R(A,B) \, B v^0 \in T_{v^0}^\perp$
    and since $x^0 \in T_{v^0}$, we get $\langle x^0, \herA v^0 \rangle = 0$. Thus, the algorithm terminates.
\item[{\textrm ii)}]
If $\dim(G_{\max}) < d$, then clearly $\dim(G_{\lambda}) < d$ for any  generalized eigen\-space $G_\lambda$ of $(\herA, B)$. Then,  $v^0 \in G_\lambda$ 
if and only if $\tilde v \in G_\lambda$ which is only possible on a zero set, since $\tilde v \sim \mathcal N(0,I_d)$. Thus, by additivity of the measure, $v^0 \not \in G$ a.s.
\end{itemize}
\end{remark}

The next lemma deals with the case $b=0$.

\begin{lemma}\label{l: eigenvectors and b}
    Let $v \in \mathbb S_B^{d-1}$
    and $x = P_v \tilde x/\| P_v \tilde x\|$, $\tilde x \sim \mathcal N(0,I_d)$.
    Then $v$  is a generalized eigenvector of $(\herA,B)$
    if and only if $b \coloneqq \langle x, \herA v \rangle = 0$ a.s.
\end{lemma}

\begin{proof}
    If $v \in G$, then we have by \eqref{helper} that $\herA v \in T_{v}^\perp$
    and since by construction $x \in T_{v}$, we get $b = 0$.
    \\
    On the other hand, if $v \not \in G$, then 
    $\dim ((\text{span} \, \herA v)^\perp \cap T_v) < d-1$, and therefore 
    \begin{equation*}
    \mathbb P_X \big(x \in T_v \cap \mathbb S^{d-1}: b = \langle x,\herA v \rangle = 0  \big) = 0.
    \tag*{\qedhere}
    \end{equation*}
\end{proof}

Further, we will need the following lemma.

\begin{lemma}\label{l: direction main_i}
Let $\lambda$ be a generalized eigenvalue of $(\herA,B)$. Then the affine subspace 
    $(v + T_{v}) \cap G_{\lambda}$
    is either empty or of dimension $\dim(G_\lambda)-1$.
\end{lemma}		

\begin{proof}        
We show 
    \begin{equation}\label{eq: intersection param}
   (v + T_{v}) \cap G_{\lambda} = \{ u \in G_{\lambda} \mid \scp{u}{B v} = 1\} 
    \end{equation}
    which yields the assertion.     
    Assume that $u \in (v + T_{v}) \cap G_{\lambda}$. Then $u = v + x$ for some $x \in T_{v}$ and
    \begin{equation}\label{eq: intersection equiv}
    \scp{u}{B v} = \scp{v}{B v} + \scp{x}{B v} = \normB{v} = 1.
    \end{equation}
    Conversely, if $u \in G_{\lambda}$ and $\scp{u}{B v} = 1$, we get
    \[
    \scp{u - v}{B v} = \scp{u}{B v} - \scp{v}{B v} = 0.
    \]
    Therefore, $u - v \in T_{v}$ and $u = v + (u - v) \in v + T_{v}$. 
\end{proof}

Now we can treat the case $\dim(G_{\max}) = d-1$.

\begin{theorem}\label{l: initialization main}
    If $\dim(G_{\max}) = d-1$, 
    then $v^1$ generated by Algorithm~\ref{alg:gRq} is a.s.\ in $G_{\max}$. 
    In other words, Algorithm~\ref{alg:gRq} terminates a.s.\ after one step.
\end{theorem}
\begin{proof}
    By Lemma~\ref{l: direction main_i}, 
    the intersection $S \coloneqq (v^0 + T_{v^0}) \cap G_{\max}$ is empty or $\dim(S) =\dim(G_{\max})-1 = d-2$.
    By \eqref{eq: intersection param}, we have that
    $$
    S = \emptyset \quad \Leftrightarrow \quad
    \scp{u}{Bv^0} = \scp{Bu}{v^0} = 0 \quad \text{for all } u \in G_{\max}.$$
     Since $B \in \pd$, it holds 
     $$W \coloneqq B G_{\max} = \dim(G_{\max}) = d-1 \quad \text{and} \quad
     \dim \left(W^\perp \right) =1.$$ 
    But by construction,  the probability that $v^0 \in W^\perp$ is zero,
    so that $S \neq \emptyset$ a.s. Consequently, we have
    $\dim(S) = d-2$ a.s. 
    
    We rewrite $S = u + T_{v^0} \cap G_{\max}$ with $u \in (T_{v^0} \cap G_{\max})^{\perp}$.
    Let $\{u_j: j=1,\ldots,d-1\}$ be an orthonormal basis of $T_{v^0}$ with $u_1 = u / \norm{u}$ 
    and $u_j \in T_{v^0} \cap G_{\max}$, $j=2,\ldots,d-1$. 
    For each direction $x \in T_{v^0} \cap \sphere^{d-1}$, we have an expansion 
    \[
    x = \scp{u_1}{x}u_1 + \sum_{j=2}^{d-1} \scp{x}{u_j} u_j 
    = \frac{\scp{u}{x}}{\norm{u}^2} u + \sum_{j=2}^{d-1} \scp{x}{u_j} u_j
    \]
    If $\scp{u}{x} \neq 0$, multiplying both sides with $\alpha \coloneqq \norm{u}^2 / \scp{u}{x}$ yields
    \[
    \alpha x
    = u + \alpha \sum_{j=2}^{d-1} \scp{x}{u_j} u_j
    \in u + T_{v^0} \cap G_{\max}.
    \]
    If $\scp{u}{x} = 0$, we get $x \in T_{v^0} \cap G_{\max}$. By Lemma~\ref{l: direction main}  we chose $x^0$ given $v^0$ uniformly at random on $T_{v^0} \cap \sphere^{d-1}$. 
    Since $\dim(T_{v^0} \cap G_{\max}) = d-2$ and $0 \in T_{v^0} \cap G_{\max}$,  we have  
    by Lemma~\ref{l: direction main_iii}  for all $v^0 \in \sphere_B^{d-1} \backslash W^\perp$ that 
    \[
    0 = \mathbb P_{X^0|V^0=v^0} (x^0 \in T_{v^0} \cap G_{\max}) .
    \]
    Then, by the law of total probability,
    \begin{align*}
    \mathbb P_{X^0}(x^0 \in T_{v^0} \cap G_{\max} ) 
     &= 
    \int_{v^{0} \in \sphere_B^{d-1}} \mathbb P_{X^0|V^0=v^0}(x^0 \in T_{v^0} \cap G_{\max}) ~ \dx \mathbb P_{V^0}(v^0) \\
    & = 
    \int_{v^{0} \in \sphere_B^{d-1} \backslash W^\perp}  \mathbb P_{X^0|V^0=v^0}(x^0 \in T_{v^0} \cap G_{\max}) ~ \dx  \mathbb P_{V^0}(v^0)  = 0. 
    \end{align*}
    Thus, $x^0 \notin T_{v^0} \cap \sphere^{d-1}$ a.s.\ and there exists $\alpha \neq 0$ such that 
    \[
    \alpha x^0 \in u + T_{v^0} \cap G_{\max} = T_{v^0} \cap (- v^0 + G_{\max} )
    \]
    and $v^0 + \alpha x^0 \in G_{\max}$ a.s. 
    By Remark~\ref{rem:1} ii), we know that $v^0$ is not a generalized eigenvector of $(\herA,B)$ a.s.\ and by
    Lemma \ref{l: eigenvectors and b},
    we have $b_0 \neq 0$ a.s. Thus, $\tau_0$ is a unique maximizer of $r(A,B,v^0 + \tau x^0)$, 
    we get $\alpha = \tau_0$.
\end{proof}
    
Finally, we deal with the case $\dim(G_{\max}) < d-1$.
 
\begin{theorem}\label{l: initialization main}
    Let $(v^k)_{k\in \NN}$ be the sequence generated by Algorithm~\ref{alg:gRq}.
    If $\dim(G_{\max}) < d-1$, 
    then $v^k$ is not a generalized eigenvector of $(\herA,B)$  
    for all $k \in \mathbb N$ a.s. 
    In other words, the algorithm does not terminate a.s.
\end{theorem}
\begin{proof}
    Let $\lambda$ be a generalized eigenvalue of $(\herA, B)$
    and consider an associated eigenspace $G_\lambda$. 
    We first show by induction 
    that $v^k \notin G_\lambda$ a.s.\ if $\dim (G_\lambda) < d-1$.
    
    By Remark~\ref{rem:1} ii),
    $v^0 \notin G_\lambda$ a.s. For the induction step, 
    assume that $v^k \notin G_\lambda$ a.s. 
    We show that $v^{k+1} \notin G_\lambda$ holds $\mathbb P_{X^{k} | V^k = v^k}$-a.s.
    
    By construction,  
    $v^{k+1} \in G_\lambda$ if  there exist some $\alpha \in \RR$ 
    and $x \in T_{v^k} \cap \sphere^{d-1}$ such that $v^k + \alpha x \in G_\lambda$. 
    We have that 
    \[
        \{v^k + \alpha x \in G_\lambda \mid \alpha \in \mathbb R, x \in T_{v^k} \cap\sphere^{d-1}\} 
        = (v^k + T_{v^k})\cap G_\lambda =:S.
    \]
    Hence, 
    all search directions $x$ that may yield $v^{k+1} \in G_\lambda$ 
    are given by 
    \[
    	D  \coloneqq \bigl\{ \tfrac{u}{\norm{u}} : u \in - v^k + S  \bigr\}.
    \] 
    Note that even if $x^k \in D$, 
    the choice of the step size $\tau_k$ may not give $v^k + \tau_k x^k \in S$.  
    By Lemma~\ref{l: direction main_i}, 
    we know that $S$ is either empty or $\dim(S) = \dim(G_\lambda)-1$. 
    
    In the first case, 
    we get immediately $v^{k+1} \notin G_\lambda$. 
    Otherwise, we have 
    $\dim(-v^k + S) = \dim(S) = \dim(G_\lambda)-1 < d-2$. 
    Clearly, it also holds $0 \notin - v^k + S$ if $v^k \notin G_\lambda$. Recall that by Lemma~\ref{l: direction main}, 
    we chose $x^k$ given $v^k$ uniformly on 
    $T_{v^k} \cap \sphere^{d-1} \simeq \sphere^{d-2}$. Thus, by Lemma~\ref{l: direction main_iii}, 
    \[
        \mathbb P_{V^{k+1} | V^k = v^k}( v^{k+1} \in G_\lambda) 
        \le  \mathbb P_{X^{k} | V^k = v^k}( x^k \in D ) = 0 
        \quad \text{if} \quad 
        v^k \notin G_\lambda.
    \]
    Since $v^k \notin G_\lambda$ a.s., 
    by the law of total probability we get
    \begin{align*}
        \mathbb P_{V^{k+1}}(v^{k+1} \in G_\lambda) 
        & = \int_{\sphere^{d-1}_B \backslash G_\lambda } \mathbb P_{V^{k+1} | V^k = v^k}( v^{k+1} \in G_\lambda ) \, \dx \mathbb P_{V^k} (v^k) = 0.
    \end{align*}
    This finishes the induction step.
    
    Now, 
    if  $\dim(G_{\lambda}) < d-1$ for all generalized eigenvalues $\lambda$, 
    then the assertion of the theorem follows by subadditivity of the measure.
    
    Otherwise, 
    there exists an eigenspace of dimension $d-1$. 
    This is only possible if there are exactly two unique eigenvalues 
    $\lambda_1, \lambda_2$ with 
    $\mathcal R(A,B) = \lambda_1 > \lambda_2$, $\dim(G_{\max})= 1$ 
    and $\dim(G_{\lambda_2}) = d - 1$. 
    By Remark~\ref{rem:1} ii), 
    we know that 
    $v^{0} \notin G$ a.s.\ and we have $r(A, B, v^0) > \lambda_{2}$ a.s. 
    By Lemma~\ref{l: eigenvectors and b}, 
    we get that $b_0 \neq 0$ a.s.\ and 
    \[
        r(A, B, v^1) 
        = \max_{\tau \in \RR} r(A,B, v^0 + \tau x^0) 
        \ge r(A, B, v^0) 
        > \lambda_{2} 
        \quad \text{a.s.}
    \]
    Therefore, 
    $v^1 \notin G_{\lambda_2}$ a.s. 
    Furthermore, 
    since $\dim(G_{\max}) = 1 < d-1$, 
    we have $v^1 \notin G_{\max}$. 
    Summarizing,  
    we get $v^1 \notin G$. 
    Continuing inductively, 
    we get that $v^{k} \notin G$ for all $k \in \NN$ a.s.
\end{proof}

\section{Convergence of the Algorithm} \label{sec:conv}

In this section, 
we investigate the convergence of $(r(A,B,v^k) )_{k\in \mathbb N}$ 
and $(v^k)_{k \in \mathbb N}$ in the case $\dim(G_{\max}) < d-1$. 
In accordance with Theorem \ref{prop:stepsize}, 
we use the notation
\begin{align}
    & a_k \coloneqq \scp{v^k}{A v^k},    
    \quad b_k \coloneqq 2\scp{x^k}{\herA v^k}, \quad
     c_k \coloneqq \scp{x^k}{A x^k},   
    \quad d_k \coloneqq \scp{x^k}{B x^k}. 
\end{align}
First, 
we quantify the change of the objective 
in terms of $\tau_k$ and $b_k$. 
In analogy to the sufficient decrease inequality \cite[Eq.~4.7]{boumal2023introduction} 
for the Riemannian gradient ascent \eqref{eq: Riemmanian gradient ascent}, 
here $|b_k|$ plays the role of the norm gradient $\norm{\grad f(v^k)}$.    

\begin{theorem}
    \label{th:update_gRq}
    Let $\dim(G_{\max}) < d-1$. 
    Then, we have $r(A,B,v^{k}) = a_k$ and
    \begin{align*}
        r(A,B,v^{k+1}) - r(A,B,v^{k})
        =  a_{k+1} - a_k
        = \tfrac{1}{2}\tau_k b_k > 0 \quad \text{a.s.}
    \end{align*}
    Furthermore, $(r(A,B,v^k))_{k\in \NN}$ converges a.s.\ and $\tau_k b_k \to 0$ 
    as $k \to \infty$ a.s.
\end{theorem}
\begin{proof}
    For any $k \in \NN$, 
    we have $1 = \normB{v^k}^2 = \scp{v^k}{Bv^k}$. Hence, 
    \[
        r(A,B,v^{k})
        = \frac{\scp{v^k}{Av^k}}{\scp{v^k}{Bv^k}} 
        = \scp{v^k}{Av^k} = a_k,
    \]
    yielding the first equality.
    On the other hand, 
    by \eqref{eq:func_gRq_abcd}, 
    we obtain    
    \begin{align*}
        r(A,B,v^{k+1}) - r(A,B,v^{k}) 
        &= \frac{a_k + \tau_k b_k + \tau_k^2 c_k}{1 + \tau_k^2 d_k} - a_k \\
        &= \frac{a_k + \tau_k b_k + \tau_k^2 c_k - (1 + \tau_k^2 d_k) a_k}{1 + \tau_k^2 d_k} \\
        & = \frac{\tau_k}{1 + \tau_k^2 d_k}(b_k + \tau_k(c_k - a_k d_k)).
    \end{align*}
    Since $\dim(G_{\max})<d-1$, 
    by Theorem~\ref{l: initialization main} for all $k \ge 0$ 
    we have $v^k \notin G$,
    and by Lemma~\ref{l: eigenvectors and b} 
    we get $b_k \neq 0$ a.s. 
    Then, $\tau_k \neq 0$ holds by Theorem~\ref{prop:stepsize} 
    and
    \[
        \tau_k b_k  
        = \tfrac{c_k - a_k d_k}{d_k} + \sqrt{\tfrac{(c_k - a_k d_k)^2}{ d_k^2} + \tfrac{|b_k|^2}{d_k}} 
        > 0 
        \quad \text{a.s.}
    \]
    Also, 
    in the proof of Theorem~\ref{prop:stepsize} 
    we have shown that $\tau_k$ satisfies \eqref{eq: derivative r_k}. 
    Thus, 
    it holds $\tau_k(c_k - a_k d_k) = \tfrac{1}{2}((\tau_k^2 d_k - 1) b_k)$ giving 
    \begin{align*}
        r(A,B,v^{k+1}) - r(A,B,v^{k})
        &= \frac{\tau_k}{1 + \tau_k^2 d_k}(1 + \tau_k^2 d_k) \tfrac{b_k}{2}
        = \frac{\tau_k b_k}{2} > 0 \quad \text{a.s.}
    \end{align*}
    Finally, 
    we note that by the above inequality $(r(A,B,v^k))_{k\in \NN}$ 
    is a monotonically increasing sequence 
    that is bounded from above by $\mathcal R(A,B)$. 
    Hence, 
    it converges a.s.\ and its increments admit $\tau_k b_k \to 0$ 
    as $k \to \infty$ a.s.     
\end{proof}

Since we have by Lemma~\ref{l: eigenvectors and b} 
that $b_k = 0$ a.s.\ if and only if $v^k$ is the generalized eigenvector of $(\herA,B)$, 
we want to ensure that $b_k$ vanishes as $k \to \infty$. 
Then, we can use $b_k^2$ as an adjusted stopping criterium. 

\begin{theorem}
    \label{th:b_k_to_zero}
    Let $\dim(G_{\max}) < d-1$. 
    Then it holds  $b_k \to 0$ as $k \to \infty$ a.s.
\end{theorem}
\begin{proof}
    Under the assumption, 
    $b_k \neq 0$ a.s.\ for all $k \in \NN$ by Theorem~\ref{l: initialization main}. 
    Then, the optimal step size $\tau_k$ from Theorem~\ref{prop:stepsize} satisfies \eqref{eq: derivative r_k}.
    Multiplying it with $b_k \neq 0$ and rearranging the terms yields 
    \begin{align}
        \label{eq:conv_b_k}
        0 < b_k^2 = \underbracket{\tau_k b_k}_{> 0} (\tau_k b_k d_k - 2 (c_k - a_k d_k)),
        \quad \forall k \in \NN
        \quad \textrm{a.s.}
    \end{align}
    Since $\tau_k b_k \to 0$ as $k\to \infty$ a.s.\ by Theorem~\ref{th:update_gRq}, 
    it remains to show that $\tau_k b_k d_k - 2 (c_k - a_k d_k)$ is bounded a.s.
    Starting with $d_k$, $\norm{x^k} = 1$ implies that
    \begin{equation*}
        d_k = \scp{x^k}{B x^k} \leq \norm{B}.
    \end{equation*}
    Therefore, using \eqref{eq: norm bounds}, we obtain
    \begin{align*}
        |a_k| & = \Bigl|\scp{v^k}{Av^k}\Bigr| 
        \le \norm{A} \norm{v^k}^2 
        \le \norm{A} \norm{B^{-1} } \normB{v^k}^2.
    \end{align*}
    Thus, Theorem~\ref{th:update_gRq} yields 
    \begin{equation}\label{eq: bound ak}
        0 < \tau_k b_k 
        = 2(a_{k+1} - a_k) 
        \le 2(|a_{k+1}| + |a_k|) 
        \le  4 \norm{A} \norm{B^{-1} } 
        \quad \text{a.s.}
    \end{equation}
    Analogously, we can show
    \[
        |c_k| = \Bigl|\scp{x^k}{A x^k}\Bigr| \le \norm{A} \norm{x}^2 = \norm{A},
    \]
    and
    \[
        \abs{c_k - a_k d_k}
        \le \norm{A}(1 + \norm{B} \norm{B^{-1} } ) 
        = \norm{A}(1 +   \kappa(B)).
    \] 
    Combining this, we arrive at 
    \begin{align}
        0 < \tau_k b_k d_k - 2 (c_k - a_k d_k)
        & \le 4 \norm{A} \norm{B} \norm{B^{-1} } + 2 \norm{A}(1 + \kappa(B)) \notag \\
        & = \norm{A}(2 + 6 \kappa(B)) \quad \text{a.s.} \label{eq: bound bsq tech}
    \end{align}
    for any $k \in \NN$. 
    Thus, by \eqref{eq:conv_b_k} 
    and Theorem~\ref{th:update_gRq},
    it follows $b_k^2 \to 0$ as $k \to \infty$ a.s.
\end{proof}

Before, we made an analogy between $|b_k|$ and $\norm{\grad f(v^k)}$. We formalize this connection in the next lemma.

\begin{lemma}    \label{l:Eb_k}
    The following relation holds true:
    \begin{align*}
        \EE_{x^k \sim X^k | V^k = v^k} [b_k^2]
        = \tfrac{4}{d-1} \norm{P_{v^k}\herA v^k}^2 
        = \tfrac{1}{d-1} \norm{\grad f(v^k)}^2.
    \end{align*}
\end{lemma}
\begin{proof}
    We rewrite the conditional expectation 
    of $b_k^2$ in terms of $x^k$ and $v^k$ as
    \begin{align*}
        \EE_{x^k \sim X^k | V^k = v^k}[b_k^2]
        &= 4 \EE_{x^k \sim  X^k | V^k = v^k}[\scp{x^k}{\herA v^k}^2] \\
        &= 4\EE_{x^k \sim  X^k | V^k = v^k}[(v^k)^\tT\herA x^k (x^k)^\tT \herA v^k] \\
        &= 4 (v^k)^\tT \herA \EE_{x^k \sim  X^k | V^k = v^k}[x^k (x^k)^\tT] \herA v^k \\
        &\hspace{-16pt}\overset{\textrm{Lemma~\ref{l: direction main}}}{=} \tfrac{4}{d-1}  (v^k)^\tT \herA P_{v^k} \herA v^k  \\
        &= \tfrac{1}{d-1} \norm{ 2 P_{v^k} \herA v^k}^2 
        = \tfrac{1}{d-1} \norm{\grad f(v^k)}^2,
    \end{align*}
    where we have used
    that the projection onto $T_{v^{k}}$
    is idempotent.
\end{proof}

Using Lemma~\ref{l:Eb_k} and Theorem~\ref{th:b_k_to_zero}, 
we show that $(r(A,B,v^k))_{k \in \NN}$ converges 
to a generalized eigenvalue of $(\herA,B)$ 
and a subsequence of $(v^k)_{k\in \NN}$ converges a.s.\ to the respective eigenspace. 
Recall that the distance of a point $x$ from a set  $S$ is defined by
\begin{equation}\label{def: distance to set}
    \dist(S, x) \coloneqq \inf_{y \in S} \; \norm{x - y}.
\end{equation}
For finite-dimensional subspaces $S$, 
the infimum is attained for some $y \in S$.

\begin{theorem}
    \label{thm:acc_AB}
    Let $\dim(G_{\max}) < d-1$.
    Then there exists a random variable $\lambda$,
    which is a generalized eigenvalue of $(\herA,B)$ a.s., 
    such that we have $\lim_{k \to \infty} r(A,B,v_k)  = \lambda$ a.s.
    Moreover, there exits a subsequence $(v^{k_j})_{j \in \NN}$ satisfying
    \[
        \dist(G_\lambda, v^{k_j}) \to 0 
        \quad \text{as} \quad 
        j \to \infty \quad \text{a.s.}
    \]
\end{theorem}
\begin{proof}
    By Theorem~\ref{th:b_k_to_zero}
    we have $b_k^2 \to 0$ for $k \to \infty$ a.s.
    Since $\abs{b_k} \leq 2 \norm{\herA} \norm{B^{-1}}^{1/2}$,
    Lebesgue's dominated convergence theorem 
    implies that $\EE[b^k] = \EE_{(x^k,v^k) \sim (X^k,V^k)}[b_k^2] \to 0$ for $k \to \infty$.
    Therefore,
    Lemma~\ref{l:Eb_k} implies 
    \begin{align*}
        \tfrac{4}{d-1} \EE_{v^k \sim V^{k}}[\norm{P_{v^k} \herA v^k}^2]
        &= \EE_{v^k \sim V^{k}}[ \, \EE_{x^k \sim X^k | V^k = v^k}[b_k^2 ] \,]
        \\
        &= \EE_{(x^k,v^k) \sim (X^k,V^k)} [b_k^2]
        \to 0
    \end{align*}
    for $k \to \infty$.
    Hence, there exists a subsequence $(v^{k_j})_j$ such that 
    $\norm{P_{v^{k_j}}\herA v^{k_j}} \to 0$ a.s.
    This, in turn, shows
    that 
    \[
        \dist(T_{v^{k_j}}^{\perp}, \herA v^{k_j})
        \to 0 \quad \text{as} \quad  j \to \infty 
        \quad \text{a.s.} 
    \]
    Thus, there exist a random sequence $\lambda_j \in \RR$ such that
    \begin{align}
        \label{eq:eigen_equation}
        \norm{ \herA v^{k_j} - \lambda_j B v^{k_j}} \to 0 \quad \text{as} \quad  j \to \infty \quad \text{a.s.}
    \end{align}
     We observe that 
    \begin{align*}
        a_{k_j} 
        = \scp{v^{k_j}}{A v^{k_j}}
        = \scp{v^{k_j}}{\herA v^{k_j}}
        &= \underbracket{\scp{v^{k_j}}{\lambda_j B v^{k_j}}}_{=\lambda_j} 
        + \underbracket{\scp{v^{k_j}}{(\herA - \lambda_j B) v^{k_j}}}_{\to 0, \ j \to \infty, \ \text{a.s.}}
    \end{align*}
    and using the a.s.\ convergence of $(a_k)_{k=0}^{\infty}$ established in Theorem~\ref{th:update_gRq}, 
    we conclude 
    \begin{align*}
        \lim_{k \to \infty} a_k = \lim_{j \to \infty} a_{k_j} = \lim_{j \to \infty} \lambda_j =: \lambda.
    \end{align*}
    Since $(\rv{v}^{k_j})_{j \in \NN} \subset \sphere^{d-1}_B$, there exists a convergent subsequence $(v^{k_{j_\ell}})_{\ell \in \NN}$ with limit $v \in \sphere^{d-1}_B$. 
    Combining it with \eqref{eq:eigen_equation} yields 
    \begin{align*}
        0 = \lim_{\ell \to \infty} \norm{ \herA \rv{v}^{k_{j_\ell}} - \lambda_{j_\ell} B\rv{v}^{k_{j_\ell}}} 
        = \norm{\herA v - \lambda Bv} 
        \quad \text{a.s.,}
    \end{align*}
    showing that $\lambda$ is a generalized eigenvalue of $(\herA, B)$ a.s.
    Finally, 
    \begin{align*}
    \norm{ \herA v^{k_j} - \lambda B v^{k_j}} 
     & \le \norm{ \herA v^{k_j} - \lambda_j B v^{k_j}} + |\lambda - \lambda_j| \norm{B v^{k_j}} \\
    & \le \norm{ \herA v^{k_j} - \lambda_j B v^{k_j}} + |\lambda - \lambda_j| \norm{B}^{1/2} \to 0,
    \quad j \to \infty
    \quad \text{a.s.}
    \end{align*}
    gives the convergence of $\dist(G_{\lambda}, \rv{v}^{k_j})$ to $0$.  
\end{proof}

Next, we show that $\lambda$ in Theorem~\ref{thm:acc_AB} 
only takes the value $\mathcal R(A,B)$ a.s.
First, 
we derive a technical result that with probability at least $p>0$,
for all $v \in \sphere^{d-1}_B$,
performing an iteration of Algorithm~\ref{alg:gRq} 
yields a vector close to a maximizer. 
Crucial here is that $p$ uniform probability, 
i.e, is independent of the choice of $v$.  

\begin{lemma}    \label{l:uniform_BS}
    Let $\varepsilon > 0$ 
    and $v, v^* \in \sphere_B^{d-1}$.
    Consider
    \begin{align*}
        \mathcal D_v 
        \coloneqq 
        \bigl\{ x \in T_{v} \cap \sphere^{d-1}
        \mid
        \;\exists \tau \in \RR : 
        R_{v}(\tau x)\in \mathbb B_{\varepsilon,B}(\pm v^*) \cap \sphere_B^{d-1}
        \bigr\},
    \end{align*}
    where $\mathbb B_{\varepsilon,B}(\pm v^*) \coloneqq \mathbb B_{\varepsilon,B}(v^*) \cup \mathbb B_{\varepsilon,B}(-v^*)$
    is the union of the $\varepsilon$-balls with respect to $\normB{\cdot}$
    around $v^*$ and $-v^*$. 
    Then, there exists $p = p(\varepsilon, d, B) > 0$, such that for all $v \in \sphere_B^{d-1}$ the surface measure of the unit sphere $\sigma_{Bv}$ in $T_{v}$ satisfies 
    \begin{align*}
        \inf_{v \in \sphere_B^{d-1}} 
        \; \sigma_{Bv}(\mathcal D_{v}) \geq p.
    \end{align*}
\end{lemma}
\begin{proof}
    The proof is analogous to \cite[Lemma~2.16]{bresch2024matrixfreestochasticcalculationoperator}
    and can be found in Appendix~\ref{sec:proofs}. 
\end{proof}

Now, we are able to prove the strictly monotone convergence 
of our sequence $(r(A,B, v^k))_{k\in \mathbb N}$
towards $\mathcal R(A, B)$
almost surely.

\begin{theorem}
    \label{thm:almost_sure_convergence}
    Let $\dim(G_{\max}) < d-1$.
    Then, it holds
    \begin{align*}
        r(A,B, v^k) \to \mathcal R(A,B),
        \quad \text{and} \quad \dist(G_{\max}, v^k)\to 0, 
        \quad \text{as} \quad 
        k \to \infty 
        \quad \text{a.s.}
    \end{align*}
\end{theorem}
\begin{proof}
    Consider
    $$
        \mathcal L 
        \coloneqq 
        \bigl\{v \in \sphere_B^{d-1} : \scp{v}{\herA v} = r(A,B,v) > \lambda_2(B^{-1}\herA)\bigr\},
    $$
    and events $E_k \coloneqq \{ V^k \not \in \mathcal L\}.$

    By monotonicity of $(r(A,B,V^{k}))_{k\in \mathbb N}$ established in Theorem~\ref{th:update_gRq} we get $E_{k+1} \subseteq E_{k}$ for any $k \in \NN$. Hence,
    \[
    \mathbb P_{V^{k+1} | V^{k} = v^k}(v^{k+1} \not \in \mathcal L) = 0 
    \quad \text{for all} \quad v^k \in \mathcal L,
    \]
    and
    \begin{align*}
    \PP(E_{k+1}) 
    & = \mathbb P_{V^{k+1}}(v^{k+1} \not \in \mathcal L)
    = \int_{\sphere_B^{d-1} \backslash \mathcal L} \mathbb P_{V^{k+1} | V^{k}=v^k}( v^{k+1} \not \in \mathcal L ) \dx \mathbb P_{V^{k}} (v^k).
    \end{align*}
    
    Let $v^* \in G_{\max} \cap \sphere^{d-1}_B$. By continuity of $r(A,B,v)$, there exists $\varepsilon >0$ such that $\mathbb B_{\varepsilon,B}(\pm v^*) \subseteq \mathcal L$. By Lemma~\ref{l:uniform_BS}, we bound
    \begin{align*}
        \mathbb P_{V^{k+1} | V^{k} = v^k}( v^{k+1} \not \in \mathcal L )
        & = \mathbb P_{X^{k} | V^{k} = v^k}( R_{v^k}(\tau_k x^k) \not \in \mathcal L) \\
        & \le \mathbb P_{X^{k} | V^{k} = v^k}( x^k \not \in \mathcal D_{v} )
        \le 1 - p.
    \end{align*}
    Consequently, we get
    \[
        \PP(E_{k+1}) 
        \le (1 - p) \mathbb P_{V^k} ( v^{k} \not \in \mathcal L ) 
        = (1 - p) \PP( E_k ),
    \]
    and, 
    inductively, $\PP(E_{k}) \le (1- p)^k \,\PP( E_0 )$.
    Hence, $\sum_{k = 0}^{\infty} \PP(E_{k}) < \infty$ 
    and by the first Borel-Cantelli lemma \cite[Theorem 7.2.2]{athreya2006measure}, 
    event $E_k$ only occur finitely many times 
    and there exists $k_0 \ge 0$ 
    such that for all $k \ge k_0$ we have $V^k \in \mathcal L$.

    By Theorem~\ref{thm:acc_AB}, 
    $\lim_{k \to \infty} r(A,B,V^k) = \lambda$ a.s. 
    for some generalized eigenvalue $\lambda$ of $(\herA,B)$. 
    On the other hand, for $k \ge k_0$
    we have that $r(A,B,V^k) > \lambda_2(B^{-1}\herA)$ 
    and $(r(A,B,V^{k}))_{k=0}^\infty$ is monotonically increasing by Theorem~\ref{th:update_gRq}. 
    As a result, 
    it holds $\lambda > \lambda_2(B^{-1}\herA)$ 
    and $\lambda = \mathcal R(A,B)$ a.s.  

    Let $u_1$, \ldots, $u_d$ 
    be the eigenvectors of $B^{-1/2}\herA B^{-1/2}$ 
    forming the basis of $\RR^{d}$ and let $\lambda_j$ 
    be the corresponding eigenvalues. 
    Since
    \[
        B^{-1/2} \herA B^{-1/2} u = \lambda u 
        \quad \text{is equivalent to} \quad
        B^{-1} \herA B^{-1/2} u = \lambda B^{-1/2} u,
    \]
    the eigenvectors of $B^{-1} \herA$ are $v_j = B^{-1/2} u_j$ 
    and the eigenvalues of $B^{-1} A$ and $B^{-1/2} \herA B^{-1/2}$ coincide. 
    Then, with $w^k \coloneqq B^{1/2} v^{k}$,
    we have
    \begin{align*}
        r(A,B,v^{k}) 
        & = \scp{v^{k}}{\herA v^{k}} 
        = \scp{w^k}{B^{-1/2}\herA B^{-1/2} w^{k}}  \\
        & = \sum_{i,\ell=1}^{d} \scp{w^{k}}{u_i} \scp{w^{k}}{u_\ell} \lambda_{\ell} \scp{u_i}{u_\ell}
        = \sum_{\ell=1}^{d} \lambda_{\ell}\scp{w^{k}}{u_\ell}^2 \\
        & \le \mathcal R(A,B) \norm{w^{k}}^2. 
    \end{align*}
    Since $\norm{w^k}^2 = \norm{B^{1/2}v^{k}}^2 = \normB{v^{k}}^2 = 1$ 
    and $r(A,B,v^{k}) \to \mathcal R(A,B)$ as $k \to \infty$, 
    we get an equality in the limit. 
    This is only possible if $\scp{w^{k}}{u_\ell}^2 \to 0$ except for $u_\ell$ 
    corresponding to the eigenvalue $\mathcal R(A,B)$.
\end{proof}

In general, we do not get a convergence of $(v^k)_{k\in \NN}$ to some $v \in G_{\max}$. 
Yet, we show it in the special case $\dim(G_{\max}) =1$.

\begin{lemma}\label{l:conv1dim}
    Let $\dim(G_{\max}) =1$. 
    Then the sequence $(v^k)_{k\in \mathbb Z}$ converges a.s.\ 
    to a generalized eigenvector of $(\herA,B)$ corresponding to $\mathcal R(A,B)$.      
\end{lemma}
\begin{proof}
    By Theorem~\ref{thm:almost_sure_convergence}, we have 
    \begin{equation*}
        \dist(G_{\max}, v^k) \to 0, 
        \quad \text{as} \quad k \to \infty \quad \text{a.s.}
    \end{equation*}
    Since $\dim(G_{\max})=1$, 
    we can write $G_{\max} = \{ \alpha v \mid \alpha \in \RR\}$, 
    where $v$ is an eigenvector with $\normB{v} = 1$ 
    corresponding to $\mathcal R(A,B)$. 
    Since $\normB{v^k} = 1$ for all $k \in \NN$, 
    the only two possible accumulation points of $v^k$ are $v$ and $-v$. 
    Hence, 
    there exists a random sequence $(e_k)_{k\in \NN} \subseteq \{-1,1\}$, 
    such that $\normB{e_k v^* - v^k } \to 0$ as $k \to \infty$. 
    Then, 
    for $0 < \varepsilon < \min\{1, \tfrac{1}{2}\kappa^{-1/2}(B)\}$, 
    there exists some random $k_{\varepsilon} \in \NN$ such that 
    \begin{equation}\label{eq: conv dim 1 tech}
        \normB{ e_k v^* - v^k } < \varepsilon,
        \quad \text{for all } k \geq k_{\varepsilon} .
    \end{equation}
    We show that the sequence $(e_k)_{k=0}^{\infty}$ 
    is constant for all $k \ge k_\varepsilon$ by contradiction. 
    Let us assume that $e_k = 1$ and $e_{k+1} = -1$, 
    for some $k > k_{\varepsilon}$. Expanding \eqref{eq: conv dim 1 tech} 
    gives
    $$
        \tfrac{1}{2} 
        < 1 - \tfrac{\varepsilon}{2} 
        < \scp{B v^*}{v^k},
        \quad \text{and} \quad 
        \scp{B v^*}{v^{k+1}} < \tfrac{\varepsilon}{2} - 1 
        < -\tfrac{1}{2}.
    $$
    On the other hand,
    following update of Algorithm~\ref{alg:gRq} and $x^k \in T_{v^k}$, we get
    \begin{equation*}
        \scp{B v^*}{v^{k+1}}
        = \frac{\scp{B v^*}{v^k + \tau_k x^k}}{\sqrt{1 + \tau_k^2 d_k}}
        = \frac{\scp{B v^*}{v^k} + \tau_k \scp{B(v^* - v^k)}{x^k}}{\sqrt{1 + \tau_k^2 d_k}}.
    \end{equation*}
    We bound the second term by
    \begin{equation*}
        \abs{\scp{B (v^* - v^k)}{x^k}}
        \le \normB{(-1)^{e_k}v^* - v^k} \normB{x^k}
        \le \varepsilon \norm{B}^{1/2},
            \end{equation*}
    which with $d_k = \scp{x^k}{B x^k} \ge \norm{B^{-1}}^{-1}$ yields the contradiction
    \begin{equation*}
        \scp{B v^*}{v^{k+1}} 
        > \frac{\scp{B v^*}{v^k}}{\sqrt{1 + \tau_k^2 d_k}} - \frac{\varepsilon|\tau_k|\lambda_1^{1/2}(B)}{\sqrt{1 + \tau_k^2d_k}}
        > - \frac{\varepsilon \lambda_1^{1/2}(B)}{\sqrt{d_k}}
        > - \frac{1}{2}. 
    \end{equation*}
    Analogously, we show that $e_k = -1$ 
    and $e_{k+1} = 1$ is impossible, 
    implying that the sequence $(e_k)_{k\in \NN}$ 
    takes the same value for $k \ge k_\varepsilon$.
\end{proof}

\section{Convergence Rates}\label{sec:convergence_rates}
In this section, we derive a convergence rates for $(b_k)_{k\in \mathbb N}$ and the error in the eigenvector relation 
$$
    \norm{\herA v^k - \langle v^k, A v^k\rangle B v^k}.
$$
We start with the sublinear convergence result for $(b_k)_{k\in \mathbb N}$.

\begin{lemma}    \label{l:min_b_k}
    Let $\dim(G_{\max}) < d-1$.
    Then, we have
    \begin{align*}
        \min_{0 \leq k \leq n} b_k^2
        \leq \tfrac{8}{n+1} \norm{A}^2 \norm{B^{-1}} (1 + 3\kappa(B))
        \quad \text{a.s.} 
    \end{align*}
\end{lemma}
\begin{proof}
    In the proof of Theorem \ref{th:b_k_to_zero}, we showed \eqref{eq:conv_b_k} and \eqref{eq: bound bsq tech}. Combining these inequalities yields
    \begin{align}
        0 
        < b_k^2 
        &= \tau_k b_k (\tau_k b_k d_k - 2(c_k - a_k d_k)) \\
        &\leq \tau_k b_k \norm{A}(2 + 6 \kappa(B)),
        \quad \text{a.s.}
    \end{align}
    By Theorem~\ref{th:update_gRq} we have $\tau_k b_k = 2(a_{k+1} - a_{k})$. 
    Therefore, using \eqref{eq: bound ak}, we obtain 
    \begin{align}
        0 < \sum_{k = 0}^n b_k^2 
        &\leq 2 \norm{A}(1+ 3\kappa(B)) \sum_{k = 0}^n \tau_k b_k 
        = 4\norm{A}(1 + 3\kappa(B)) \sum_{k = 0}^n (a_{k+1} - a_k) \\
        & = 4\norm{A}(1 + 3\kappa(B))(a_{n+1} - a_0)
        \leq 8 \norm{A}^2 \norm{B^{-1}}(1 + 3\kappa(B)) \quad \text{a.s.}
        \label{xxx}
    \end{align}
    On the other hand, it holds $\sum_{k = 0}^n b_k^2 \geq (n+1)\min_{0 \leq k \leq n} b_k^2$,
    which yields the assertion.
\end{proof}

Further, we will apply the following uniform bound, 
whose proof is given in the \ref{sec:proofs}.

\begin{lemma}
    \label{lem:bound_B_kappa}
    Let $B \in \pd, v \in \sphere_B^{d-1}$ and $y \in \RR^d$.
    Then it holds
    \begin{equation*}
        \norm{(I_d - Bvv^\tT)y} \leq \sqrt{\kappa(B)} \norm{(I_d - \tfrac{Bv (Bv)^\tT}{\norm{Bv}^2}) y}.
    \end{equation*}
\end{lemma}

The above lemmas allow us to establish an a.s. sublinear convergence rate 
to the critical point of $f$. 
Since such points are generalized eigenvectors of $(\herA,B)$, 
this is in line with previously established a.s. convergence of $(v^k)_{k \ge 0}$ 
to $G_{\max}$ in Theorem~\ref{thm:almost_sure_convergence}. 
We now also provide the convergence rate for the minimal squared residual
\[
    \min_{k = 0, \ldots, n} \norm{\herA v^k - \langle v^k, A v^k\rangle  B v^k}^2.
\]

\begin{theorem}
    \label{thm:conv_rate_eigen_equation}
    Let $\dim(G_{\max}) < d-1$.
    Then, 
    we have $\grad f(v^k) \to 0$ as $k \to \infty$ a.s. and
    \[
        \min_{k = 0, \ldots, n} \| \grad f(v^k) \|^2
        \le \frac{8(d-1)}{n+1} \norm{A}^2 \norm{B^{-1}} (1 + 3\kappa(B)) \quad \text{a.s.}
    \]
    Moreover, the minimal squared residual satisfies
    \[
        \min_{k = 0, \ldots, n} \norm[\Big]{\herA v^k - \scp{v^k}{A v^k} B v^k}^2 
        \le \frac{8(d-1)}{n+1} \norm{A}^2 \norm{B^{-1}} \kappa(B) (1 + 3\kappa(B)) 
        \quad \text{a.s.}
    \]
\end{theorem}
\begin{proof}
    For the gradient, we combine Lemma \ref{l:Eb_k} with \eqref{xxx} to a.s. get    
    \[
        \sum_{k=0}^{n} \| \grad f(v^k) \|^2 
        = (d-1) \sum_{k = 0}^n \EE_{x^k \sim X^k | V^k = v^k} [b_k^2]  
        \leq 8(d-1) \norm{A}^2 \norm{B^{-1}}(1 + 3\kappa(B)).
    \]
    Thus, 
    series $\sum_{k=0}^{\infty} \| \grad f(v^k) \|^2$ converge a.s. 
    and the summand $\| \grad f(v^k) \|^2$ vanishes a.s. 
    Furthermore, it a.s. holds 
    \begin{align*}
        \min_{k = 0, \ldots, n} \| \grad f(v^k) \|^2 
        & \le \frac{1}{n+1} \sum_{k=0}^{n} \| \grad f(v^k) \|^2 \\
        &  \leq \frac{8(d-1)}{n+1} \norm{A}^2 \norm{B^{-1}}(1 + 3\kappa(B)).
    \end{align*}
    To finish the proof, we observe that
    \begin{align*}
        \herA v^{k} - \scp{v^{k}}{\herA v^{k}} B v^{k}
        = (I_d - B v^{k}(v^{k})^\tT) \herA v^{k}
    \end{align*}
    and, consequently, by Lemma~\ref{lem:bound_B_kappa}
    \begin{align*}
        \norm[\Big]{\herA v^{k} - \scp{v^{k}}{\herA v^{k}} B v^{k}}^2
        & \le \norm{(I_d - B v^{k}(v^{k})^\tT) \herA v^{k}} \\
        & \le \kappa(B) \norm{(I_d - \tfrac{Bv (Bv)^\tT}{\norm{Bv}^2}) \herA v^{k}}^2 \\
        & = \kappa(B) \norm{\grad f(v^k)}^2.
    \end{align*}
    Combining this inequality with the established bound 
    for the norms of the gradients concludes the proof.
\end{proof}

Unfortunately, 
neither rate provides a quantitative convergence estimate for the distance of the iterate to the leading generalized eigenspace, i.e. of $\dist(v^k, G_{\max})$. 
In particular, we have 
\[
    \norm{\herA v^{k} - \mathcal R(A,B) B v^{k}}
    \le
    \norm{\herA v^{k} - \scp{v^{k}}{\herA v^{k}} B v^{k}}
     + (\mathcal R(A,B) - r(A,B,v^k)) \norm{B v^k},
\]
which implies that for establishing the rate of convergence to $G_{\max}$, 
we would need a rate of convergence for $\mathcal R(A,B) - r(A,B,v^k)$, 
and we were not able to derive it.

\section{Connection to the Zeroth-Order Methods}\label{sec: zeroth-order link}

We can also reinterpret Algorithm~\ref{alg:gRq} as a zeroth-order method.

\begin{theorem}\label{th: main stoch gradient}
    It holds
    \[
        b_k x^k = 2\widehat{\grad}_1 f(v^k)
        \quad \text{and} \quad
        \EE_{x^k \sim X^{k} | V^k = v^k}[b_k x^k] 
        = \tfrac{1}{d-1} \grad f(v^k),
    \]
    where $\grad f(v^k)$ is the Riemannian gradients \eqref{eq: riemannian gradient special}
    of $f(v^k) = \scp{v^k}{A v^k}$ on $\sphere^{d-1}_B$ 
    and $\widehat{\grad}_1 f$ is its zeroth-order one-sample approximation \eqref{eq: m-sample gradient} 
    with scaling parameter $\mu = \tau_k$. 
    Consequently, 
    Algorithm~\ref{alg:gRq} can be viewed as a zero-order variant
    of Riemannian gradient ascent \eqref{eq: zero order ascent}.
\end{theorem}

\begin{proof}
    If $b_k \neq 0$, by Theorem~\ref{prop:stepsize}we get $\sign(b_k) = \sign(\tau_k)$ and the vectors $b_k x^k$ and $\tau^k x^k$ are aligned. Furthermore, we have $r(A,B,v^k) = f(v^k)$ and Theorem~\ref{th:update_gRq} yields
    \[
        b_k x^k = \frac{2}{\tau_k} \frac{\tau_k b_k}{2} x^k
        = 2 \frac{f( R_{v^{k}}(\tau_k x^k)) -  f(v^{k})}{\tau_k} x^k 
        = 2 \, \widehat{\grad}_1 f(v^k).
    \]
    If $b_k = 0$, then by Theorem~\ref{prop:stepsize},
    we set $\tau_k = 0$ and 
    $b_k x^k = 0 = 2\widehat\grad_1 f(v^k)$ 
    with $\mu = 0 = \tau_k$. Here $0/0$ is interpreted as $0$.
    By Lemma~\ref{l: direction main}, we obtain 
    \begin{align*}
        \label{eq: EE_b_x}
        & \EE_{x^k \sim X^{k} | V^k = v^k}[b_k x^k]
        = \EE_{x^k \sim X^{k} | V^k = v^k}[ x^k \cdot 2 (x^k)^\tT \herA v^k ]  \\
        & \quad = 2 \, \EE_{x^k \sim X^{k} | V^k = v^k} [x^k(x^k)^\tT ] \herA v^k 
        = \tfrac{2}{d-1}P_{v^k} \herA v^k = \tfrac{1}{d-1} \grad f(v^k).
    \end{align*}
\end{proof}

\begin{remark}[Improved Gradient Estimate]\label{discussion}
    Theorem~\ref{th: main stoch gradient} shows that $b_k x^k$ serves as an estimate of the Riemannian gradient.
    We can reduce the variance of this estimate by employing multiple samples. 
    Namely, by sampling independently $x^{k,i}$, $i = 1, \ldots, m$ 
    and computing the corresponding $b_{k,i}$, 
    the vector $\bar x^k = \frac{1}{m} \sum_{i=1}^{m} b_{k,i} x^{k,i}$ 
    is again an estimate of $\grad f(v^k)$ with (conditional) variance reduced by $1/m$. 
    Furthermore, 
    by Theorem~\ref{th: main stoch gradient} $\bar x^k = 2 \widehat{\grad}_m f(v^k)$.
    This observation is the main motivation for Algorithm~\ref{alg:gRq m-sample}. 
    Unlike zeroth-order gradient ascent \eqref{eq: zero order ascent}, 
    we select step sizes $\tau_k$ optimally, 
    leading to a drastically better performance 
    as shown in Section \ref{sec:zeroth-order}.
\end{remark}

\section{Numerical Results}\label{sec:num}

In this section, 
we provide simulations to support our theoretical findings and compare our algorithm with
established methods.
We start with a proof-of-concept experiment 
for the one-sample algorithm in Subsection~\ref{sec:proof-of-concept}
Then, 
in Subsection~\ref{sec:riemannan-aprox}, 
we study how well $x^k$ from the $m$-sample Algorithm~\ref{alg:gRq m-sample} 
approximates the Riemannian gradient in terms of $m$.
We face the problem of ill-conditioned matrices $B$ 
in the generalized Rayleigh quotient in Subsection~\ref{sec:ill-conditioned}. 
Finally, we compare our algorithm with the zeroth-order method   
from \cite{li2023stochastic}, see \eqref{eq: zero order ascent},
in Subsection~\ref{sec:zeroth-order},
and with the deterministic and averaged stochastic Gen-Oja method \cite{bhatia2018genijastreaminggeneralizedeigenvector}
in Subsection~\ref{sec:gen-oja}. 

All algorithms are implemented in Python and the code is publicly available
\footnote{%
\url{https://github.com/JJEWBresch/ZerothOrderGeneralizedRayleighQuotient}.}.
The experiments are performed on an off-the-shelf MacBook Pro 2020 
with Intel Core i5 (4‑Core CPU, 1.4~GHz) and 8~GB~RAM.

\subsection{Proof-of-Concept Example} \label{sec:proof-of-concept}
We apply our approach 
on random Gaussian matrices 
$A \in \RR^{d\times d}$ and random positive definite matrices $B = (\tilde B + d \cdot I_d)^\tT(\tilde B + d \cdot I_d)$
generated by Gaussian matrices $\tilde B \in \RR^{d\times d}$ for $d \in \{10,50,100,500\}$. 
We generate 50 random problems in the described manner
and report the average of their \emph{relative quotient error}
\[
    \textrm{RQE}_k = \frac{\mathcal R(A,B) - r(A,B,v^k)}{\mathcal R(A,B)},
\]
and the \emph{minimal squared residual} in the eigenvector equation
\[
    \textrm{MSQR}_k = \min_{n \leq k} \norm{\herA v^n - \scp{v^n}{\herA v^n} B v^n}^2,
\]
as well as the quantity $|b^k|$, 
whose decay allows us to track the convergence 
of the proposed algorithms as shown in Section~\ref{sec:convergence_rates}.
The results are depicted in Figure~\ref{fig:exp_1}.
We observe linear decay in \textrm{RQE} for all dimensions, 
which slows down as the dimension $d$ increases. 
Since random matrices have a nonzero eigengap 
between the two largest eigenvalues with probability one, 
this behavior can be linked to linear convergence rates discussed, 
e.g., in \cite{alimisis2021distributed, alimisis2024geodesic}.
Both, \textrm{MSQR} and $|b_k|$, indicate sublinear convergence 
as stated in Theorem~\ref{thm:conv_rate_eigen_equation} and Lemma~\ref{l:min_b_k}. 
Figure~\ref{fig:exp_1} also shows that incorporating more samples $m>1$
significantly improves the performance in all metrics, 
and, in particular, for a large $m$, the value $|b_k|$ decays linearly.

\begin{figure}[b!]
    \resizebox{\linewidth}{!}{%
    \begin{threeparttable}
    \begin{tabular}{c c c c c}
        & ${10}$ & ${50}$ & ${100}$ & ${500}$ \\
        \rotatebox{90}{\hspace{1cm} $\abs{b_k}$ \hspace{1.35cm} $\textrm{MSQR}_k$ \hspace{2cm} $\textrm{RQE}_k$}
        & \includegraphics[width=0.5\linewidth, clip=true, trim=10pt 180pt 30pt 80pt]{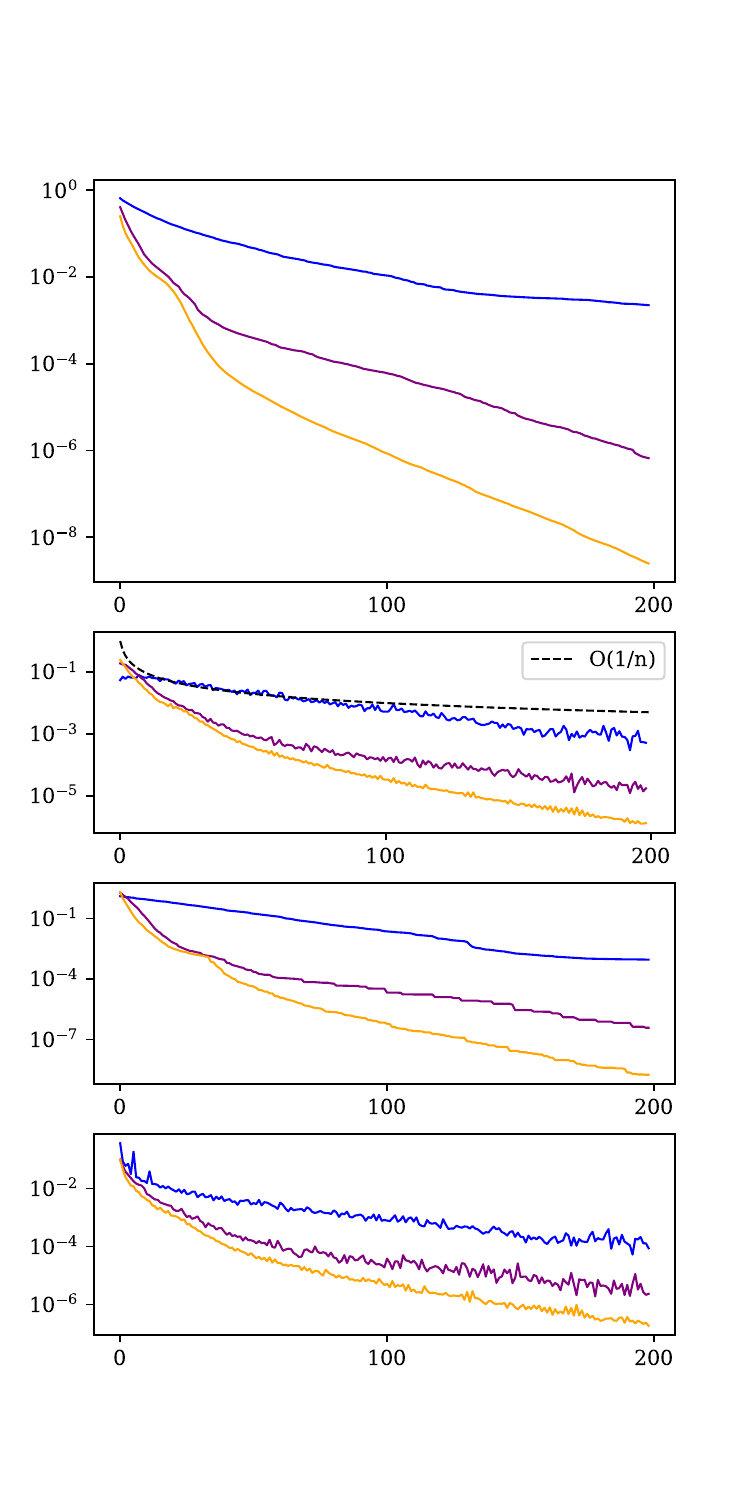}
        & \includegraphics[width=0.5\linewidth, clip=true, trim=10pt 180pt 30pt 80pt]{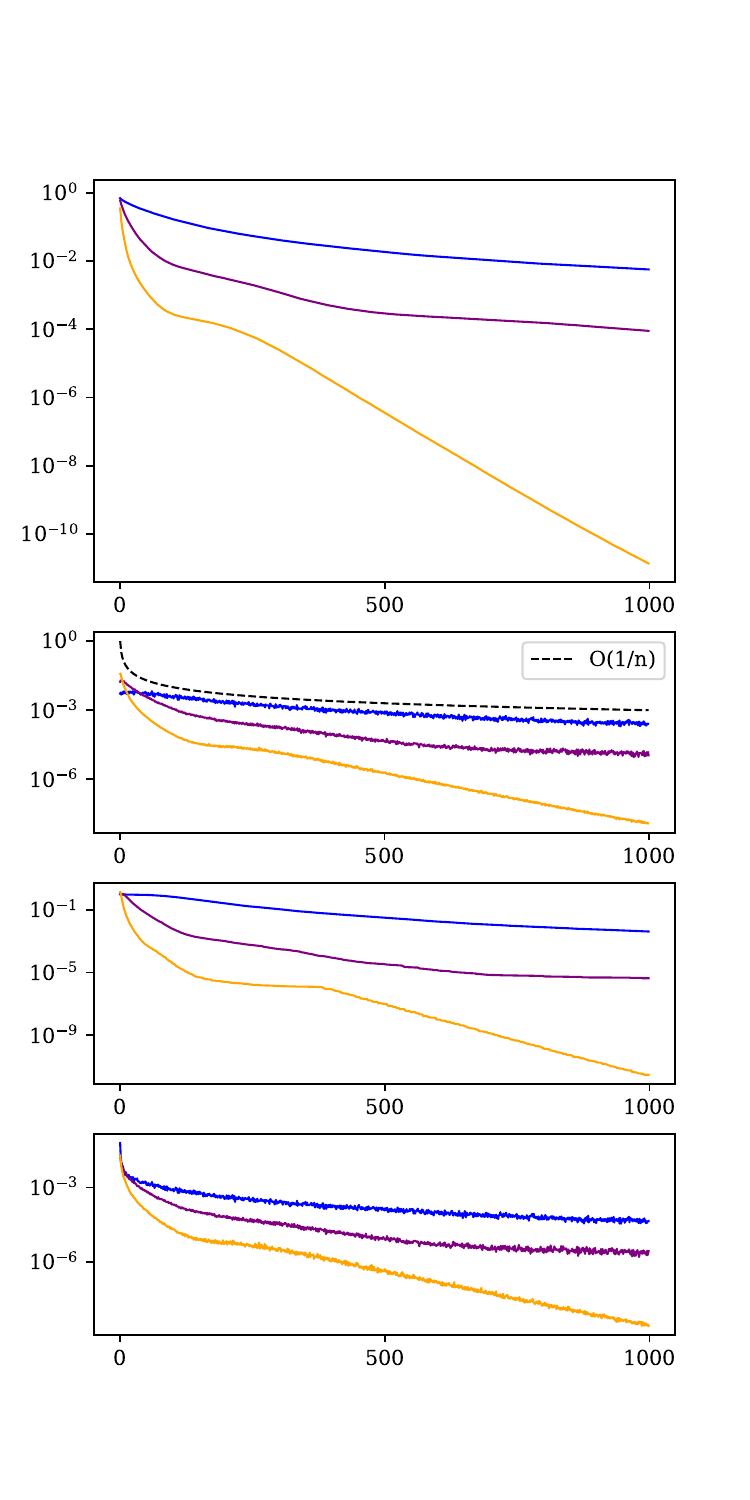}
        & \includegraphics[width=0.5\linewidth, clip=true, trim=10pt 180pt 30pt 80pt]{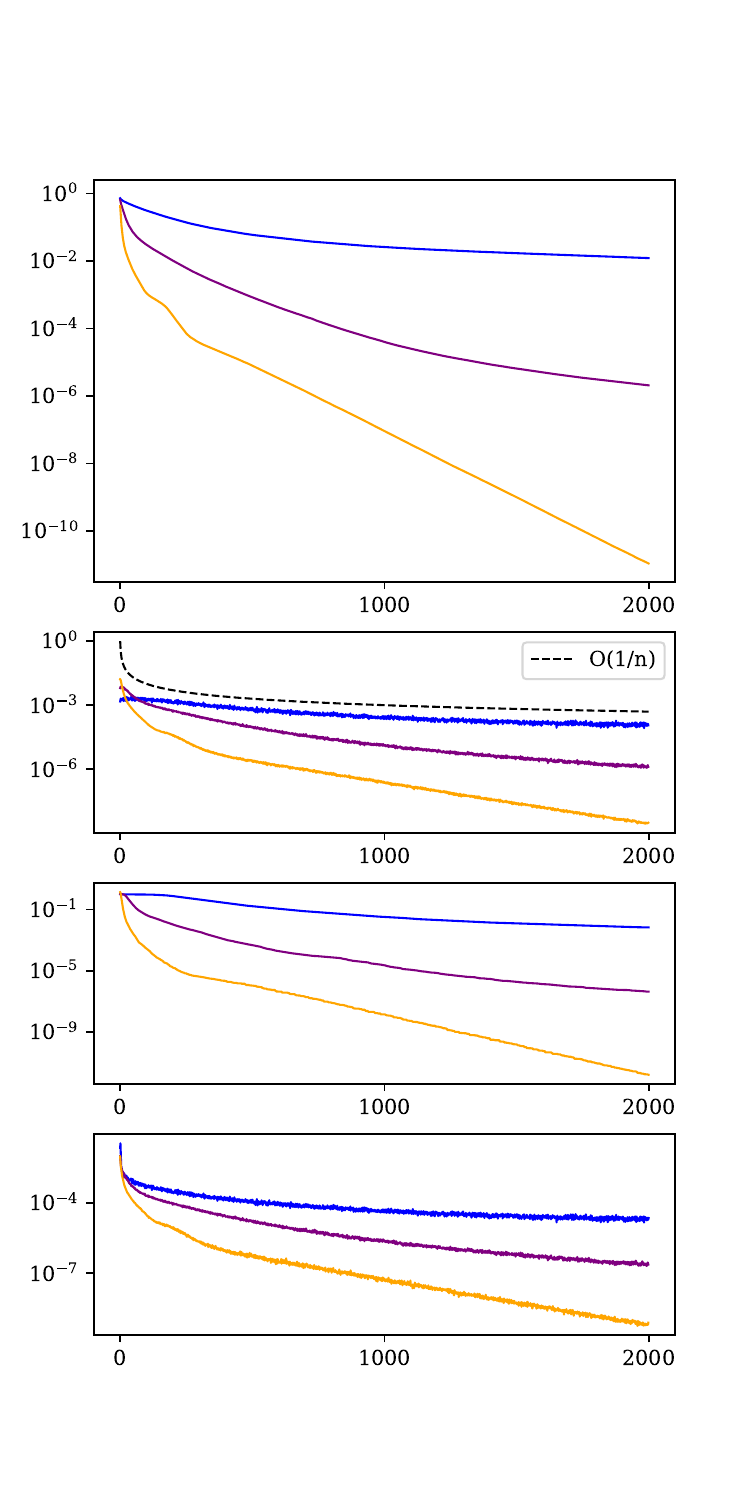}
        & \includegraphics[width=0.5\linewidth, clip=true, trim=10pt 180pt 30pt 80pt]{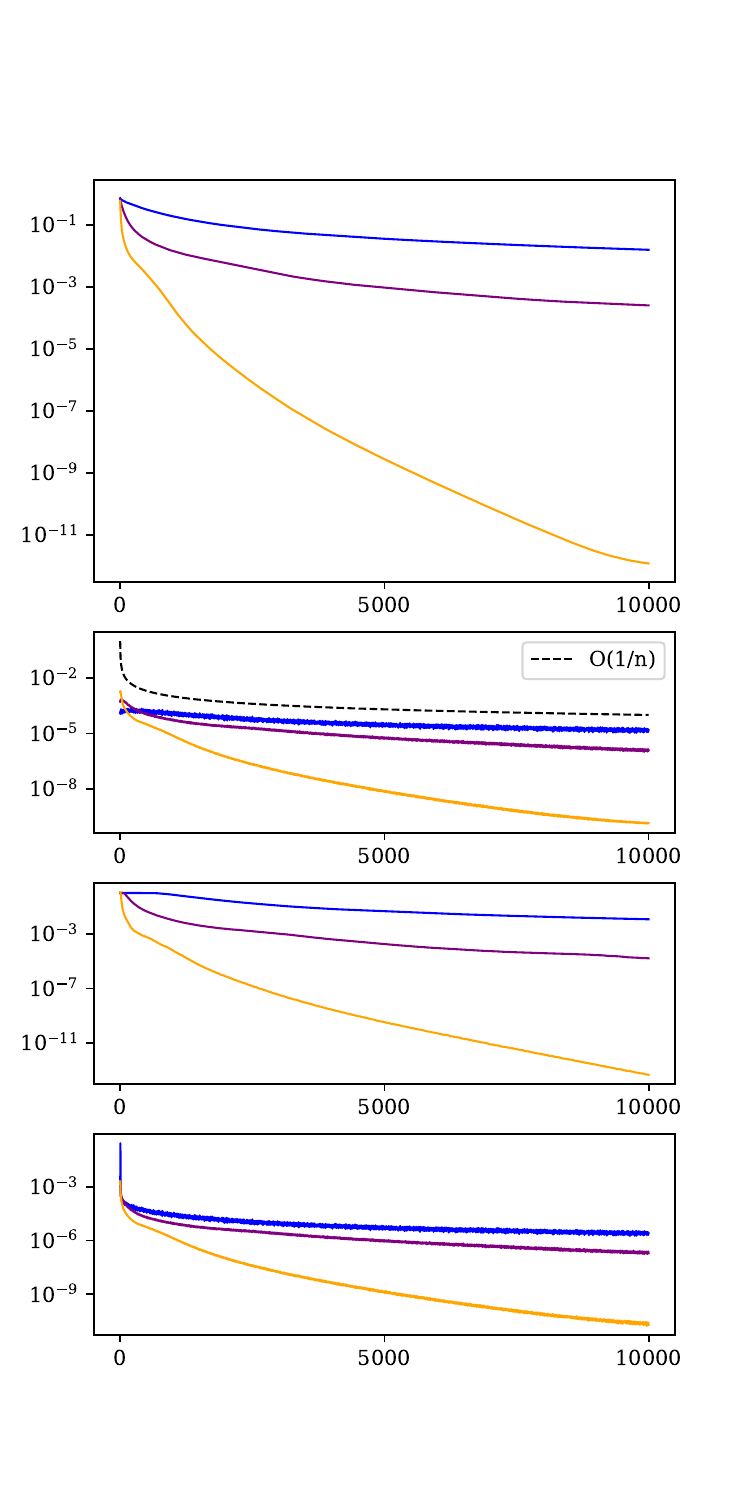} \\
        & Iterations $k$ & Iterations $k$ & Iterations $k$ & Iterations $k$
    \end{tabular}
    \begin{tablenotes}
    \item[] \scalebox{1.75}{\hspace{4.2cm} \scriptsize $m=1$ \lineblue \qquad  $m=10$ \linepurple \qquad $m=100$ \lineorange}
    \end{tablenotes}
    \end{threeparttable}
    }
    \caption{Convergence of our algorithm for different sizes $d\in \{10,50,100,500\}$
    and $m \in \{1,10,100\}$.
    }
    \label{fig:exp_1}
\end{figure}

\subsection{Estimation of the Riemannian Gradient}   \label{sec:riemannan-aprox}

Next we investigate how well $|b_k|^2$ approximates $\tfrac{1}{d-1}\norm{\grad f(v^k)}^2$ 
as established in Theorem~\ref{th: main stoch gradient}. 
Fifty trials are performed for randomly generated matrices 
as in Subsection~\ref{sec:proof-of-concept}.
Since we observed similar behavior of the mean quantities in all dimensions, 
we only report $d = 100$. 
Figure~\ref{fig:comp_g_b} (left) indicates that the two values are well-aligned, 
with the one-sample estimator being quite precise. 
As $m$ increases, $|b_k|^2$ provides an upper bound for the actual Riemannian gradient.
This motivates taking $\abs{b_k}^2$ or its running average as a stopping criterion.
In Figure~\ref{fig:comp_g_b} (right), 
the approximation error $\norm{(d-1) x^{k} - \grad f(v^k)}$ is reported. 
For $m=1$, the difference does not significantly change 
with an increasing number of iterations, 
and the variance of the estimator reduces drastically as $m$ increases, 
supporting the theoretical results of Theorem~\ref{th: main stoch gradient}.
Notably, the error behaves similarly to the norm of the Riemannian gradient, 
aligning with results \cite[Lemma~7]{li2023stochastic},
and the numerical observations of the improved convergence behavior 
with respect to the number of iterations from Subsection~\ref{sec:proof-of-concept}.

\begin{figure}
    \resizebox{\linewidth}{!}{%
    \begin{threeparttable}
    \begin{tabular}{c c}
        \includegraphics[width=0.7\linewidth, clip=true, trim=60pt 15pt 430pt 30pt]{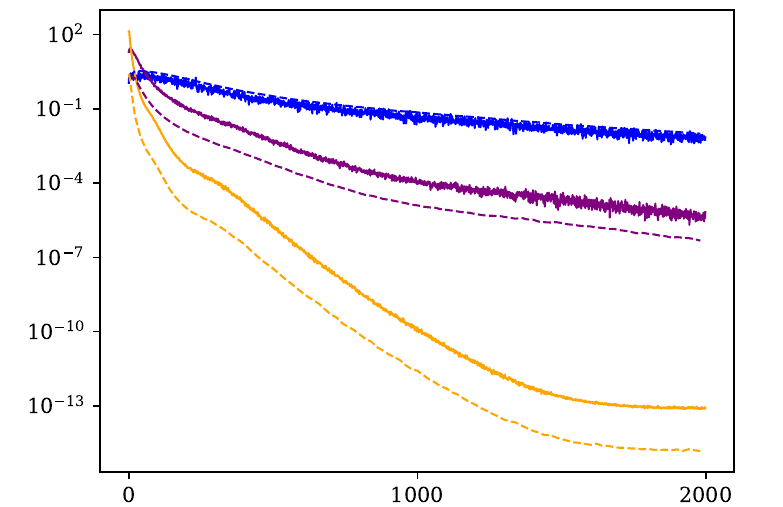}
        & \includegraphics[width=0.7\linewidth, clip=true, trim=430pt 15pt 60pt 30pt]{pdf_plots/gaussian_matrices_comp_grad.pdf} \\
        Iterations $k$ & Iterations $k$
    \end{tabular}
    \begin{tablenotes}
    \item[] \scalebox{1.4}{\hspace{2.75cm} \scriptsize $m=1$ \lineblue \qquad $m=10$ \linepurple \qquad $m=100$ \lineorange}
    \end{tablenotes}
    \end{threeparttable}
    }
    \caption{Error estimation towards Riemannian gradient for $d=100$ and 
    $m \in \{1,10,100\}$.
    \textit{Left}: Comparison of $\abs{b_k}^2$ (solid lines) 
    with $\tfrac{1}{d-1}\norm{\grad f(v^k)}^2$ (dashed lines).
    \textit{Right}: Error $\norm{(d-1) x^{k} - \grad f(v^k)}$.
   }
    \label{fig:comp_g_b}
\end{figure}

\subsection{Convergence Time for Ill-Conditioned $B$}   \label{sec:ill-conditioned}

As observed in Subsection~\ref{sec:proof-of-concept} 
and~\ref{sec:riemannan-aprox}, choosing a larger $m$ improves the convergence behavior.
Next, we study the convergence if $B$ is ill-conditioned.
Therefore, 
we vary the spectrum of $B$ and its condition number $\kappa(B)$, 
and study how it affects the performance of our algorithm.
We fix the dimension $d = 100$ and set eigenvalues of $B$ to $\lambda = 10^{p}$,
where $p$ is uniformly sampled on $(0,q)$ for $q \in \{1,2,3\}$. 
Then, we generate the eigenvectors as columns of a random unitary matrix $Q \in \RR^{d\times d}$ 
by sampling random Gaussian matrices $\tilde B \in \RR^{d\times d}$ 
and performing their QR-decomposition. 
Finally, we set $B = Q \diag(\lambda) Q^\tT$.
In contrast to the setup of Figure~\ref{fig:exp_1} and~\ref{fig:comp_g_b}, 
we execute Algorithm~\ref{alg:gRq m-sample} for a fixed number 
of $(2q-1) \cdot 1000$ iterations 
and rescale the $x$-axis proportionally to the average runtime in seconds. 
The resulting mean quantities over 50 runs are shown in Figure~\ref{fig:exp_2_time}.  

Firstly, 
we observe that despite larger arithmetic complexity
for increasing $m$, vectorization by \texttt{numpy} 
significantly speeds up the implementation such that 
there is almost no difference in runtime between  $m=1$ and $m = 10$, 
and for $m=100$ takes at most twice the amount of time.
Within a fixed time period, the improvement of \textrm{RQE} increases with $m$.

\begin{figure}[t!]
    \resizebox{\linewidth}{!}{%
    \begin{threeparttable}
    \begin{tabular}{c c c c}
        & $q = 1$ & $q = 2$ & $q = 3$ \\
        \rotatebox{90}{\hspace{1.75cm} $\textrm{RQE}_t$}
        & \includegraphics[width=0.5\linewidth, clip=true, trim=10pt 425pt 20pt 80pt]{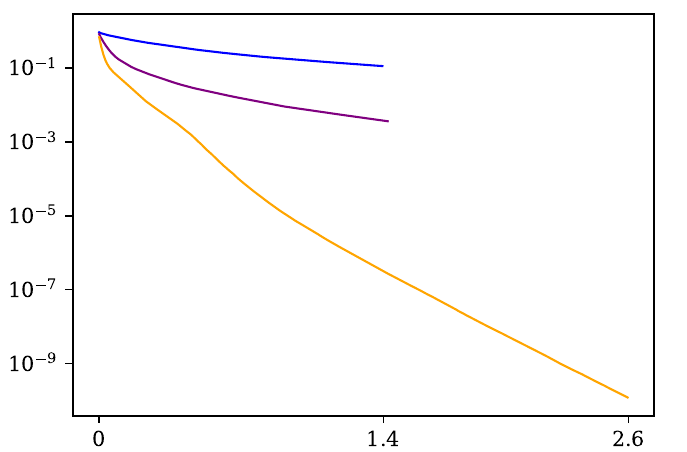}
        & \includegraphics[width=0.5\linewidth, clip=true, trim=10pt 425pt 20pt 80pt]{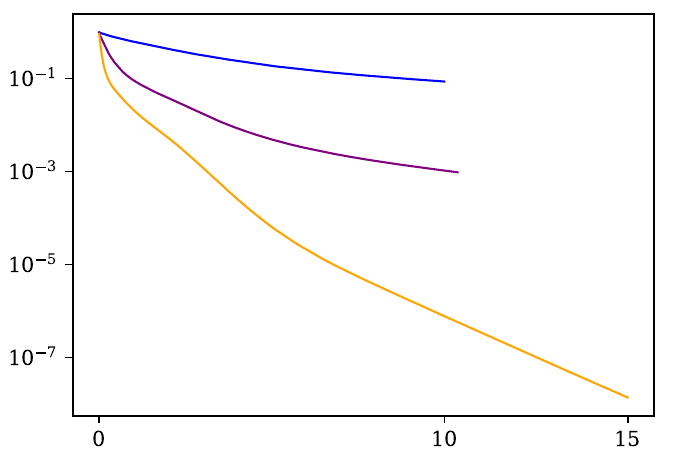}
        & \includegraphics[width=0.5\linewidth, clip=true, trim=10pt 425pt 20pt 80pt]{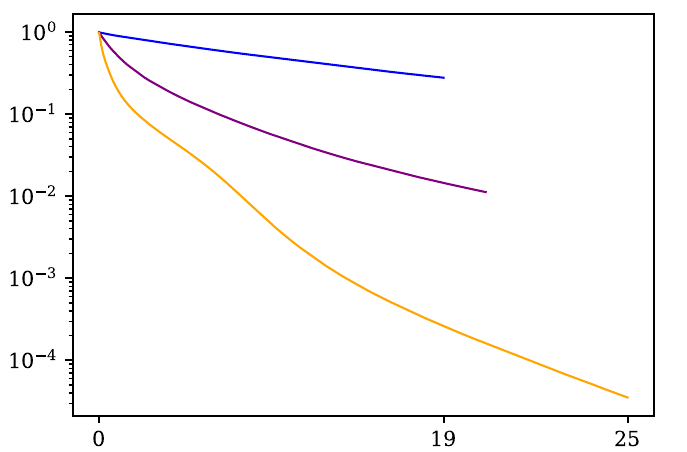}\\
        \rotatebox{90}{\hspace{1cm} $\abs{b_t}$}
        & \includegraphics[width=0.5\linewidth, clip=true, trim=10pt 180pt 20pt 425pt]{pdf_plots/uniform_matrices_average_0_1_100_A_acs_time.pdf}
        & \includegraphics[width=0.5\linewidth, clip=true, trim=10pt 180pt 20pt 425pt]{pdf_plots/uniform_matrices_average_0_2_100_A_acs_time.pdf}
        & \includegraphics[width=0.5\linewidth, clip=true, trim=10pt 180pt 20pt 425pt]{pdf_plots/uniform_matrices_average_0_3_100_A_acs_time.pdf} \\
        & time $t$ in sec. & time $t$ in sec. & time $t$ in sec.
    \end{tabular}
    \begin{tablenotes}
    \item[] \scalebox{1.15}{\hspace{5cm} $m=1$ \lineblue \qquad $m=10$ \linepurple \qquad $m=100$ \lineorange}
    \end{tablenotes}
    \end{threeparttable}
    }
    \caption{Convergence of our algorithm 
    for ill-conditioned
    $B$ with $\kappa(B) \approx 10^q$ for $q = 1,2,3$ (left to right) and $m \in\{1,10,100\}$ for a fixed number of $(2q-1)\cdot 1000$ iterations.  
        }
    \label{fig:exp_2_time}
\end{figure}

\begin{figure}[b!]
    \resizebox{\linewidth}{!}{%
    \begin{threeparttable}
    \begin{tabular}{c c c c}
        & $q = 1$ & $q = 2$ & $q = 3$ \\
        \rotatebox{90}{\hspace{0.5cm} Time $t$ in sec.}
        & \includegraphics[width=0.5\linewidth, clip=true, trim=10pt 5pt 20pt 20pt]{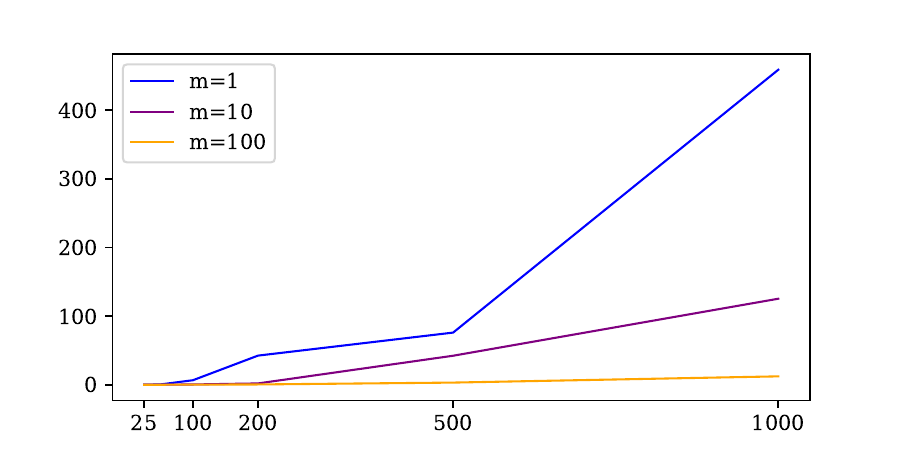}
        & \includegraphics[width=0.5\linewidth, clip=true, trim=10pt 5pt 20pt 20pt]{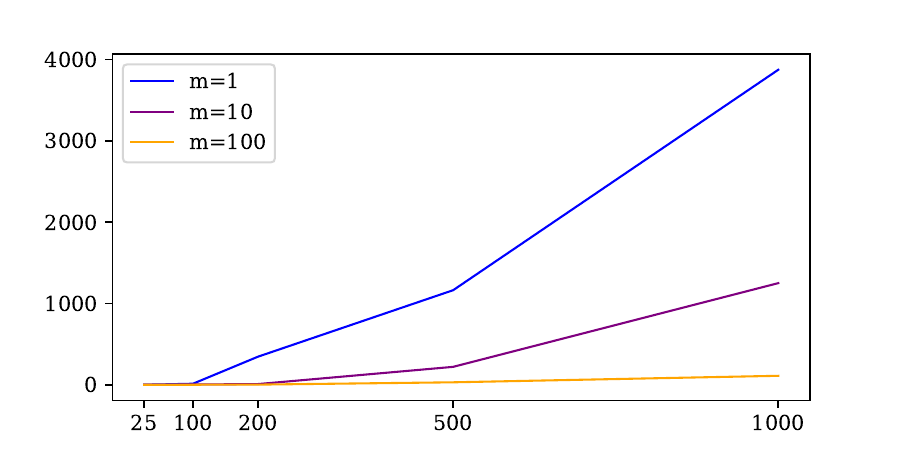}
        & \includegraphics[width=0.5\linewidth, clip=true, trim=10pt 5pt 20pt 20pt]{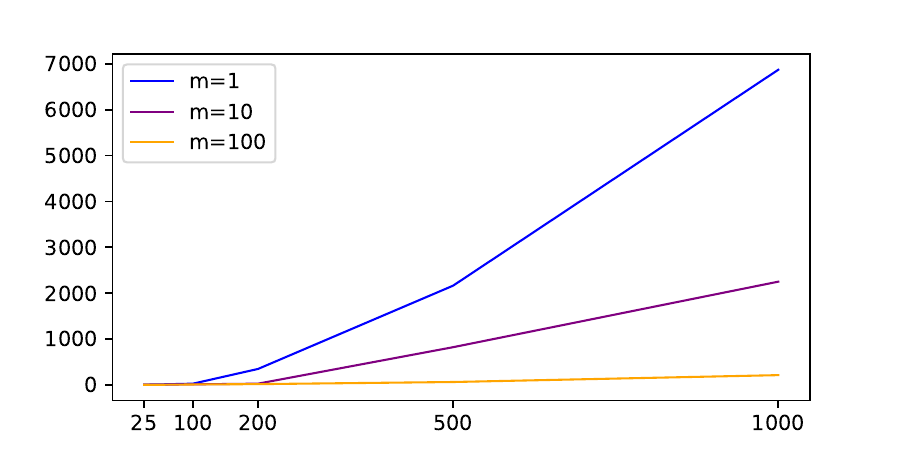}\\
        & Dimension $d$ & Dimension $d$ & Dimension $d$
    \end{tabular}
    \begin{tablenotes}
    \item[] \scalebox{1.6}{\hspace{3cm} \scriptsize $m=1$ \lineblue \qquad $m=10$ \linepurple \qquad $m=100$ \lineorange}
    \end{tablenotes}
    \end{threeparttable}
    }
    \caption{Runtime of our algorithm for $m \in \{1,10,100\}$
    for the matrices from Figure~\ref{fig:exp_2_time} for varying dimension $d$.
    The methods are stopped whenever  $\textrm{RQE} < 10^{-2}$ or after $100\cdot d$ iterations.  
    }
    \label{fig:exp_1_time}
\end{figure}

Secondly, we see that a larger $\kappa(B) \approx 10^{q}$ 
strongly increases the computation time required 
to reach the same error level.
To investigate this further, 
we vary the dimension $d$ and track the time to reach \textrm{RQE} smaller than $0.01$
in Figure~\ref{fig:exp_1_time}. 
Since the algorithm for $m=1$ is quite slow, 
for dimensions $d > 100$, in this case, 
it is stopped after $100d$ iterations if the stopping criterion is not reached. 
We observe that the expected quadratic dependency between time 
and dimension is governed by the $\mathcal O(d^2)$ complexity 
of the matrix-vector products.
Moreover, Figure~\ref{fig:exp_1_time} shows that including more samples 
leads to much faster convergence and lower computation time.

\subsection{Comparison to Zeroth-Order Method in \cite{li2023stochastic}}     \label{sec:zeroth-order}

We now compare our algorithm with \textit{zeroth-order Riemannian gradient ascent} 
(\texttt{ZO-RGA}) \cite[Algorithm~A.1]{li2023stochastic}, 
in \eqref{eq: zero order ascent}  with $m=100$-sample gradient estimators. 
Two variants of \texttt{ZO-RGA} are considered, namely 
\begin{itemize}
    \item[i)] with constant step size $L = \norm{A}(1 + \kappa(B))$, and
    \item[ii)] with Armijo-Goldstein backtracking line search, 
    see \cite[Subsection~4.2]{boumal2023introduction}
\end{itemize}
The step size in our implemented versions is larger 
than the theoretical bounds reported in Theorem~\ref{th: zero-order convergence}, 
as smaller step sizes yielded unsatisfactory performance.  

\begin{figure}[t!]
    \resizebox{\linewidth}{!}{%
    \begin{threeparttable}
    \begin{tabular}{c c c c c}
        & ${d = 10}$ & ${d = 50}$ & ${d = 100}$ & ${d = 500}$ \\
        \rotatebox{90}{\hspace{1.25cm} $\norm{\grad f(v^k)}$ \hspace{2cm} $\textrm{RQE}_k$}
        & \includegraphics[width=0.5\linewidth, clip=true, trim=10pt 30pt 10pt 50pt]{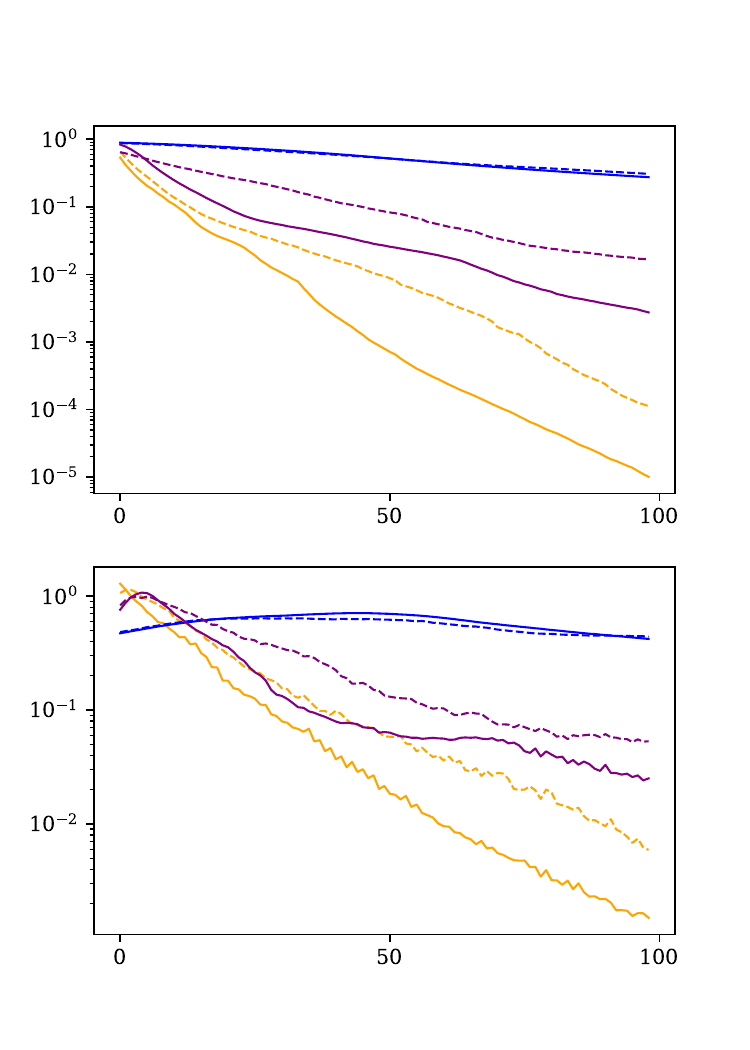}
        & \includegraphics[width=0.5\linewidth, clip=true, trim=10pt 30pt 10pt 50pt]{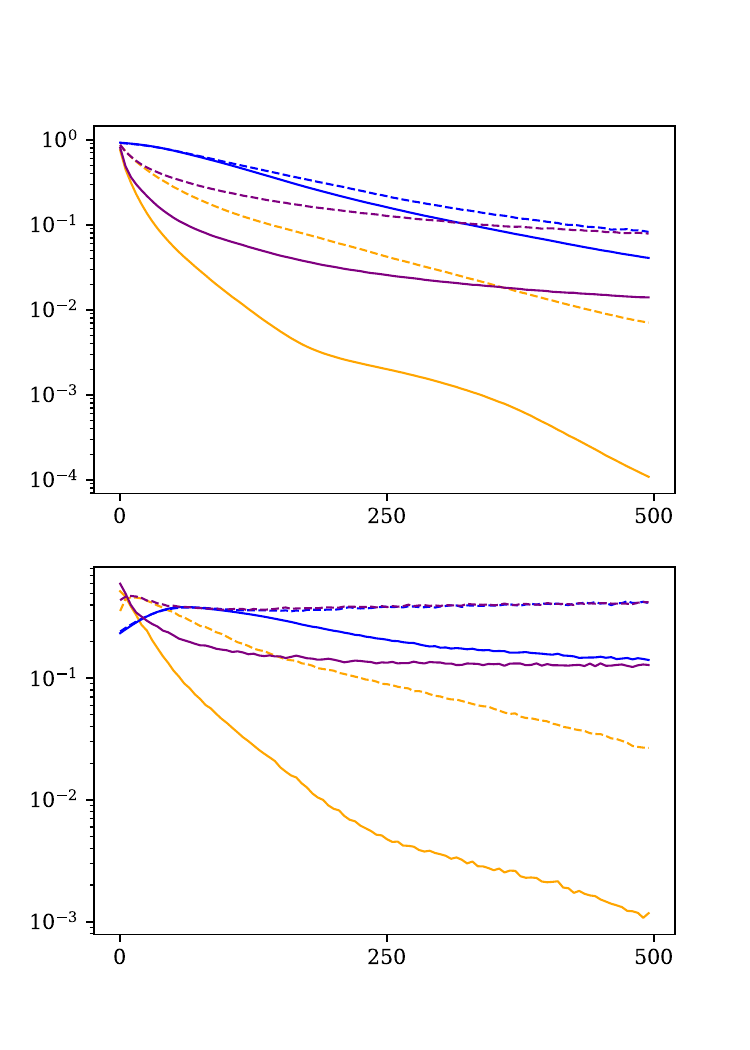}
        & \includegraphics[width=0.5\linewidth, clip=true, trim=10pt 30pt 10pt 50pt]{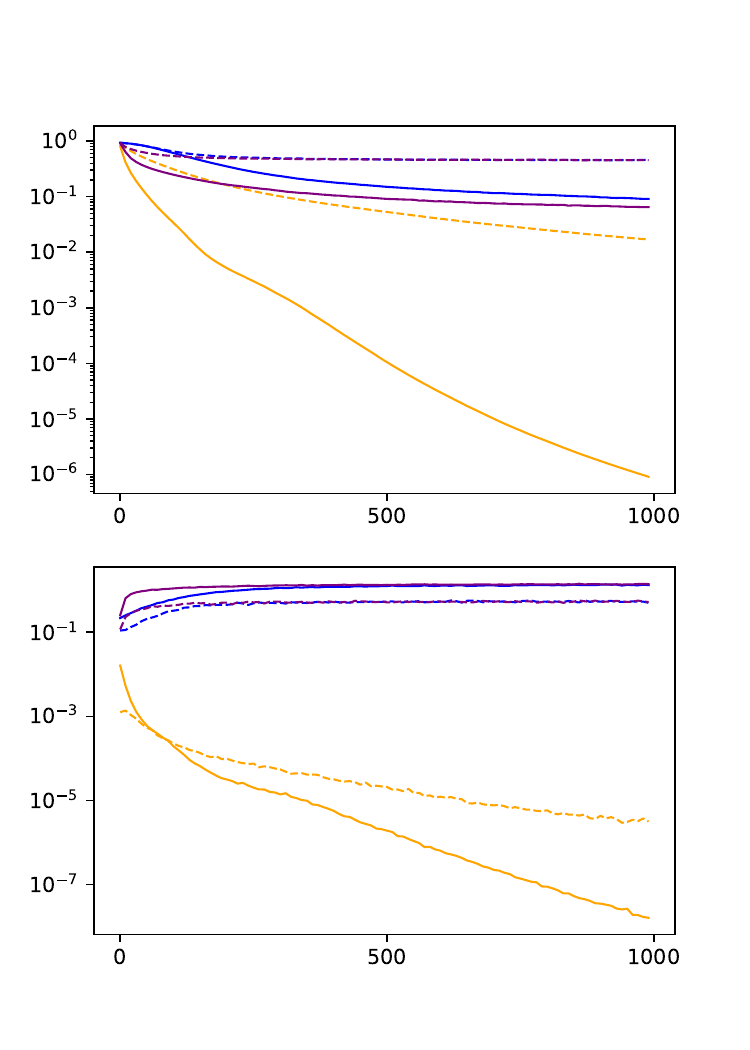}
        & \includegraphics[width=0.5\linewidth, clip=true, trim=10pt 30pt 10pt 50pt]{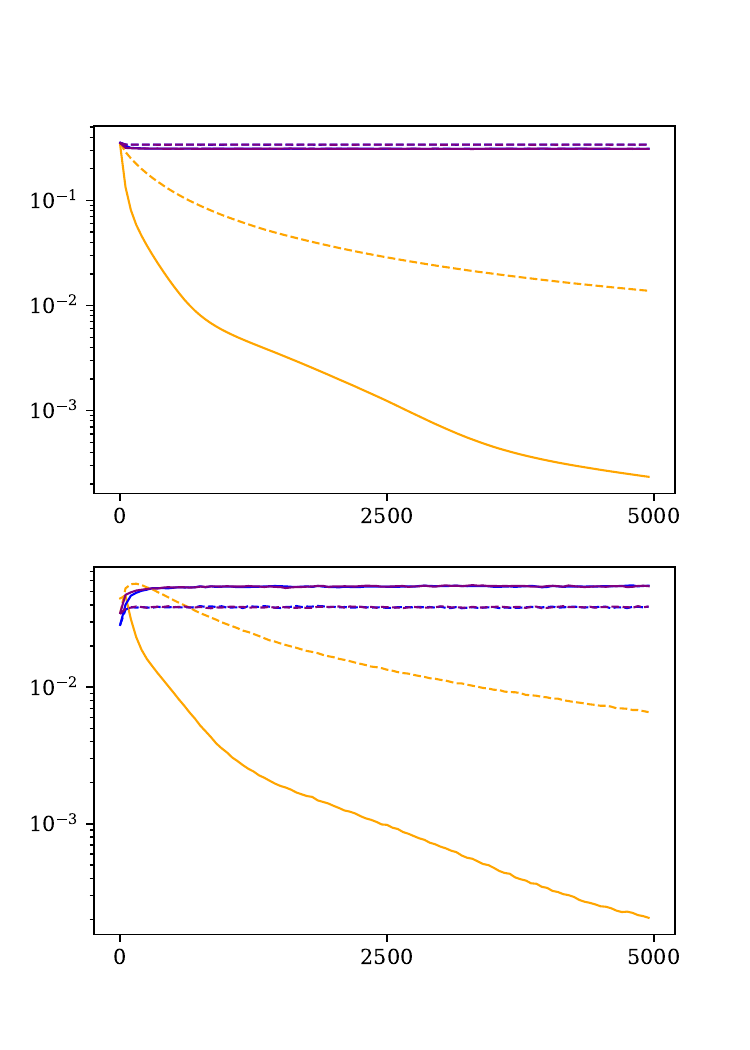} \\
        & Iterations $k$ & Iterations $k$ & Iterations $k$ & Iterations $k$
    \end{tabular}
    \begin{tablenotes}
    \item[] \scalebox{1.6}{\hspace{3cm} Our algorithm \lineorange \hspace{2cm} \texttt{ZO-RGA} i) \lineblue \hspace{2cm} \texttt{ZO-RGA} ii) \linepurple}
    \end{tablenotes}
    \end{threeparttable}
    }
    \caption{ 
    Comparison of our algorithm with the two variants of \texttt{ZO-RGA}  for $m=100$ and $d \in\{10,50,100,500\}$.
    Comparison of $m=10$ (dashed lines) and $m=100$ (solid lines).
    }
    \label{fig:exp_3}
\end{figure}

We aim to solve the \textit{generalized operator norm} problem  
\begin{align} \label{eq:ggRq}
    \argmax_{v \in \sphere^{d-1}} \; \frac{\scp{v}{Av}}{\scp{v}{Bv}}
    = \argmax_{v \in \sphere^{d-1}} \; \frac{\norm{\tilde A v}}{\norm{\tilde B v}},
\end{align}
where $A = \tilde A^\tT \tilde A$ and $B = \tilde B^\tT \tilde B$
with $\tilde A \in \RR^{d\times d}$, $\tilde B \in \RR^{2d \times d}$ 
being random Gaussian matrices to apply our proposed methods. 
Note that $A,B \in \pd$ with probability one. 
The results of both \texttt{ZO-RGA} algorithms in comparison to our algorithm 
for dimensions $d \in \{10,50,100,500\}$ and $m \in \{10,100\}$ 
are reported in Figure~\ref{fig:exp_3}.
In smaller dimensions,
i.e., $d \in \{10,50\}$, \texttt{ZO-RGA} i) performs the worst 
and does not benefit much from greedier $m$-sample gradient estimators, 
see Figure~\ref{fig:exp_3} with $m=10$ (dashed lines) 
compared to $m=100$ (solid lines). 
Notably,
\texttt{ZO-RGA} ii) yields better estimation results, 
but our algorithms still outperform it. 
For $d\in \{100,500\}$, in contrast to our algorithm, 
the \textrm{RQE} decays for both variants of \texttt{ZO-RGA} similarly very slowly.
This indicates the strength of selecting a clever $\tau_k$ 
in our Algorithm~\ref{alg:gRq m-sample}.

\subsection{Karhunen-Lo\`eve Problem}    \label{sec:gen-oja}

In the last experiment, we compare our approach with Gen-Oja methods for the Karhunen-Loève problem.
Here $A \in \RR^{300\times 300}$ is a covariance matrix 
built from an RBF kernel on a one-dimensional grid, 
and $B \in \RR^{300\times 300}$ is a mass matrix with diagonal trapezoidal weight. 
We refer to \cite{bhatia2018genijastreaminggeneralizedeigenvector} 
for more details on the problem.
Similarly to our method, 
both the deterministic and the averaged noisy Gen-Oja implementations 
in \cite{bhatia2018genijastreaminggeneralizedeigenvector} 
require only products with $A$ and $B$.
The resulting estimated eigenfunction, 
as well as the $\sin_B^2$-error defined by
\begin{equation*}
    \sin_B^2(v, v_{\textup{true}})
    \coloneqq 1 - 
    \tfrac{\scp{v}{B v_{\textup{true}}}}{\normB{v}\cdot\normB{v_{\textup{true}}}}
\end{equation*}
for our algorithm and both Gen-Oja methods 
are shown in Figure~\ref{fig:exp_4}.
We observe that the deterministic Gen-Oja method 
admits only an approximation of the solution in the $\sin_B^2$-error,
whereas the stochastic version does not find one. 
While not depicted in Figure~\ref{fig:exp_4}, 
we increased the number of iterations for both Gen-Oja methods up to 2000, 
but this did not yield any improvement. 
It can be seen that the generalized eigenfunction is only roughly approximated. 
We also note that the decrease of the $\sin_B^2$-error of our algorithm with $m=1$ 
is slow and the resulting eigenfunction exhibits high-frequency artifacts. 
In contrast,
our algorithm with a larger $m$ reduces the error much faster 
and provides a better visual match to the ground-truth eigenfunction.  

\begin{figure}[t!]
    \resizebox{\linewidth}{!}{%
    \begin{tabular}{c c}
    \rotatebox{90}{\hspace{2cm} $\sin_B^2$}
    & \includegraphics[width=\linewidth, clip=true, trim=30pt 23pt 10pt 24pt]{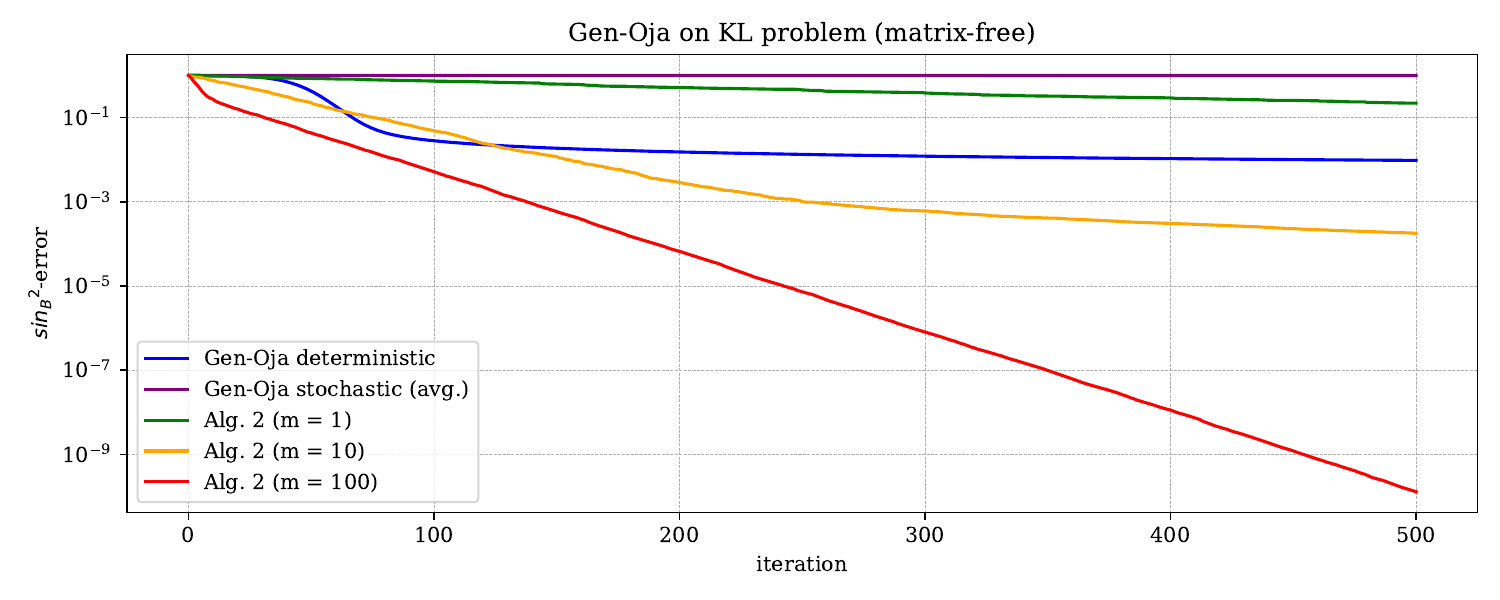} \\
    \multicolumn{2}{r}{\includegraphics[width=1.055\linewidth, clip=true, trim=25pt 8pt 57pt 24pt]{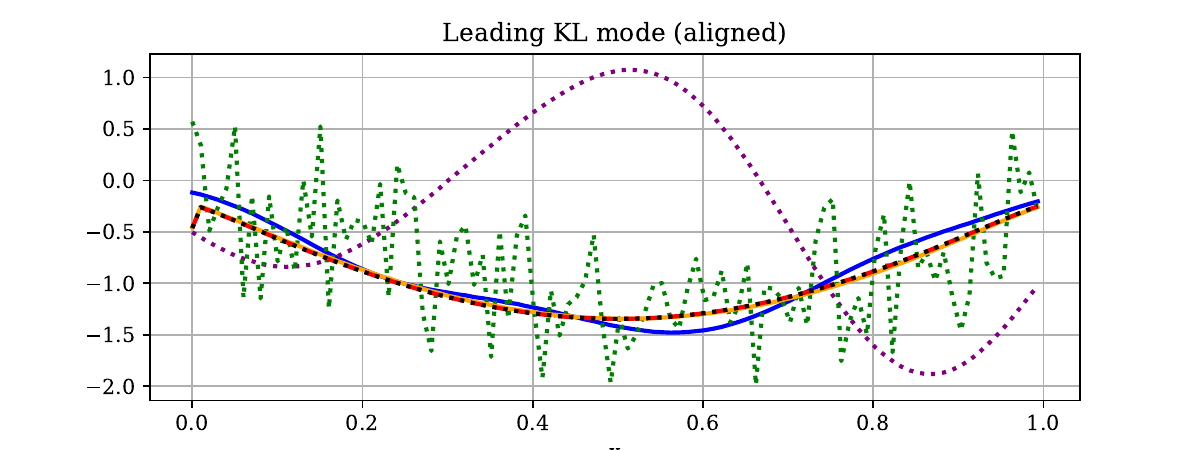}}
    \end{tabular}
    }
    \caption{%
    Solution of a Karhunen–Loève problem.
    \textit{Top:} 
    The $\sin_B^2$-error 
    for the Gen-Oja methods and our algorithm.
    \textit{Bottom:} The approximations of the generalized eigenfunction (black dotted) after 500 iterations. 
    }
    \label{fig:exp_4}
\end{figure}

\subsection{The generalized Rayleigh-Ritz method as subproblem}\label{sec:rayleigh-Ritz}

In Remark~\ref{rem:Sub-Rayleigh quotient problem} 
we saw that our optimal step size selection from Theorem~\ref{prop:stepsize} 
is the Rayleigh-Ritz method \cite{Rayleigh2011,Ritz1909} 
for two-dimensional $\text{span}\{v^k, x^k\}$.
Now,
inspired by the averaged update direction procedure in Algorithm~\ref{alg:gRq m-sample} 
we can also use Rayleigh-Ritz method 
for the $(m+1)$-dimensional hyperplane $\text{span}\{v^k, x^{1,k}, ..., x^{m,k}\}$. 
Here, $v^k \in \sphere_B^{d-1}$ is the current iterate 
and $\{x^{1,k},...,x^{m,k}\} \subset T_v \cap \sphere^{d-1}$ are sampled 
as in Algorithm~\ref{alg:gRq m-sample}.
Hence, 
we define
\begin{equation*}
    W_m^k \coloneqq [v^k, x^{1,k}, ..., x^{m,k}], 
    \quad 
    A_m^k \coloneqq (W_m^k)^{\tT} A W_m^k
    \quad \text{and} \quad 
     B_m^k \coloneqq (W_m^k)^{\tT} B W_m^k,
\end{equation*}
and search for the leading generalized eigenvector $w^k \in \RR^{m+1}$ of $(A_m^k, B_m^k)$.  

Then, 
we set the next iterate $v^{k+1} = \tilde v^{k+1} / \norm{\tilde v^{k+1}}_B$ 
with $\tilde v^{k+1} \coloneqq W_m^k w^k$, 
the optimal linear combination in $\text{span}\{v^k, x^{1,k}, ..., x^{m,k}\}$.  
Notably, there is no closed-form optimal solution, 
and finding it constitutes a subroutine 
that is solved numerically via \texttt{scipy.linalg.eigh}.

We compare the proposed Rayleigh-Ritz-based method 
with Algorithm~\ref{alg:gRq m-sample} for Rayleigh quotient maximization 
in dimensions $d \in \{100,250\}$ and for $m \in \{1,10,50,100\}$ samples. 
We run both algorithms for $2000$ iterations and $50$ randomly sampled matrices 
following the construction from Section~\ref{sec:proof-of-concept}.
The resulting mean values of the RQE and MSQE as functions of runtime 
are reported in Figure~\ref{fig:exp_5}.

\begin{figure}
    \resizebox{\linewidth}{!}{%
    \begin{threeparttable}
    \begin{tabular}{c c c}
        & $d = 100$ & $d = 250$ \\
        \rotatebox{90}{\hspace{1.7cm} MSQE$_t$ \hspace{3.3cm} RQE$_t$}
        & \includegraphics[width=0.75\linewidth, clip=true, trim=10pt 23pt 10pt 40pt]{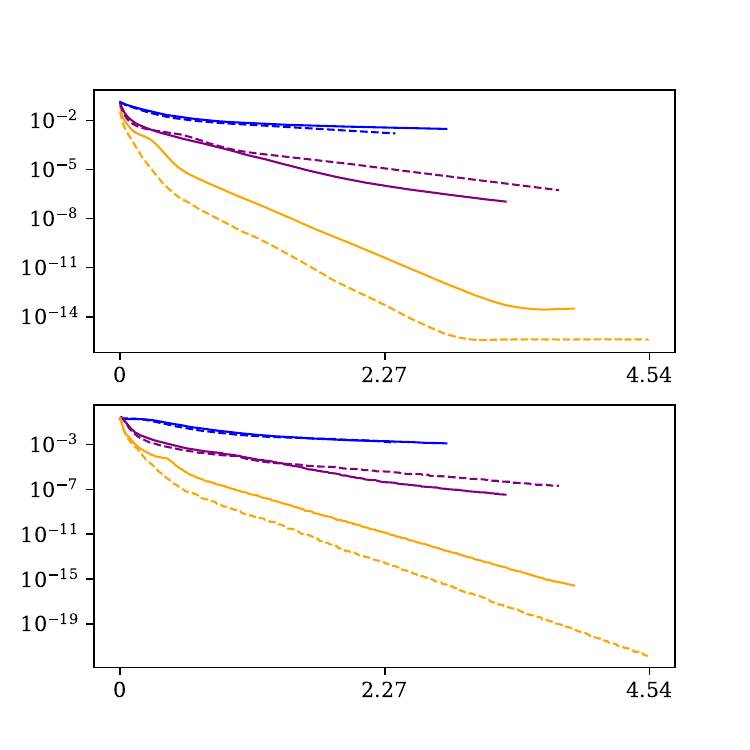}
        & \includegraphics[width=0.75\linewidth, clip=true, trim=10pt 23pt 10pt 40pt]{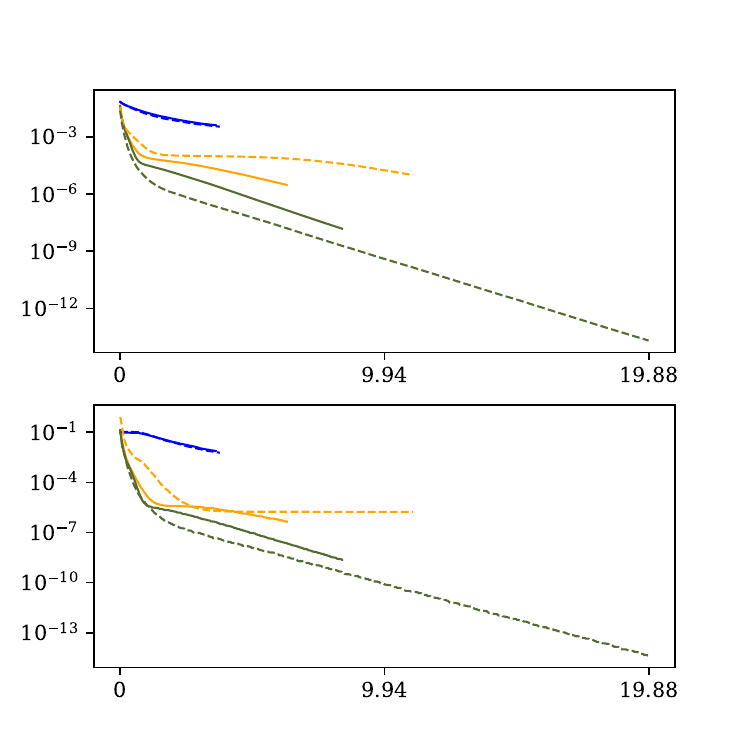} \\
        & Iterations $t$ & Time $t$
    \end{tabular}
    \begin{tablenotes}
    \item[] \scalebox{1.6}{\hspace{1.5cm} \scriptsize $m=1$ \lineblue \qquad $m=10$ \linepurple \qquad $m=50$ \lineorange \qquad $m=100$ \linegreen}
    \end{tablenotes}
    \end{threeparttable}}
    \caption{Comparison in iteration (left) and time (right) of Algorithm~\ref{alg:gRq m-sample} (solid line) and the generalized Rayleigh-Ritz baseline method (dashed line) for $m \in \{1,5,10\}$.}
    \label{fig:exp_5}
\end{figure}

We observe that the Rayleigh-Ritz method 
outperforms our proposed $m$-sampling technique
from Algorithm~\ref{alg:gRq m-sample} just for the case $m = 50$ 
when $d = 100$ and $m = 100$ when $d = 250$. 
However, for $m = 10$ and $m = 50$, respectively, Algorithm~\ref{alg:gRq m-sample} 
performs better and there is a trade-off between the dimensions $d$ and $m$.
Since the computational complexity for finding the leading eigenvector 
up to a machine precision is $\mathcal O(m^2 \log m)$ 
and construction of matrices $A_m^k$, $B_m^k$ 
requires $\mathcal (d^2 m + m^2 d)$, 
the total computational complexity for solving the subproblem 
is $\mathcal O (d^2 m + m^2 d + m^2 \log m)$, 
dominated by $\mathcal O(d^2 m)$ for $d$ much larger than $m$. 
In contrast, 
the complexity of constructing $x^k$ in Algorithm~\ref{alg:gRq m-sample} 
is $\mathcal O(dm)$. 
Figure~\ref{fig:exp_5} implies that we need fewer costly Rayleigh-Ritz iterations 
to reach the same RQE as lighter iterations of Algorithm~\ref{alg:gRq m-sample}.   
For the special case $m = 1$,
both methods are equivalent 
as explored in Remark~\ref{rem:Sub-Rayleigh quotient problem}, 
and the differences in the plots come from the randomization.

\bibliographystyle{elsarticle-num} 
\bibliography{literature}

\appendix
\section{Proofs}    \label{sec:proofs}

\noindent
\textbf{Theorem~\ref{th: zero-order convergence}}
Let $f$ be defined by \eqref{problem_riemann}
and $L \ge 2 \norm{\herA}(1 + \kappa(B))$. 
Then,
the sequence $(v^k)_{k=0}^{\infty}$ generated by \eqref{eq: zero order ascent} 
with $\tau_k = 1/\big(2(d+4) L \big)$ and scaling parameters $\mu_k$ 
atisfying $\sum_{k \in \NN} \mu_k^2 < \infty$ fulfills $\grad f(v^k) \to 0$ a.s.\ 
as $k\to \infty$ 
and there exists a constant $C > 0$ depending on $L$ and $d$ such that  
\[
    \min_{k=0,\ldots,n} \EE [\norm{\grad f(v^k)}^2] 
    \le \frac{8(d+4) L}{n+1}[\mathcal R(A,B) - f(v^0) + C \sum_{k=0}^{\infty} \mu_k^2].
\]
\begin{proof}
    We note that maximization of $f$ is equivalent to minimization of $-f$. 
    Let consider Gaussian random vectors $X^k = (X_1^k, \ldots, X_m^k)$ 
    from $\widehat{\grad}_m f(v^k)$, 
    given by \eqref{eq: m-sample gradient}, 
    in the $k$th iteration of zeroth-order gradient ascent \eqref{eq: zero order ascent}.
    The proof of \cite[Theorem A.1]{li2023stochastic} establishes inequality
    \[
        \EE_{x^k \sim X^k \mid V^K = v^k}[-f(v^{k+1})] 
        \le -f(v^k) - \tfrac{\tau_k}{4} \norm{\grad f(v^k)}^2 + \mu_k^2 C, 
    \]
    where $C = \frac{L}{16 (d+4)^2}[(d+3)^4 + (d+6)^4 + (d+6)^3]>0$. 
    We add $\mathcal R(A,B)$ on both sides 
    to get nonnegative terms $\mathcal R(A,B) - f(v^k)$,
    \begin{equation}\label{eq: zeroth order proof tech}
        \EE_{x^k \sim X^k \mid V^K = v^k}[\mathcal R(A,B)-f(v^{k+1})] 
        \le \mathcal R(A,B)-f(v^k) - \tfrac{\tau_k}{4} \norm{\grad f(v^k)}^2 + \mu_k^2 C, 
    \end{equation}
    Since $\sum_{k =0}^{\infty} \mu_k^2 < \infty$, 
    we can apply near-supermartingale convergence results by Siegmund and Robbins \cite{robbins1971convergence} 
    giving that the series $\sum_{k = 0}^{\infty} \norm{\grad f(v^k)}^2$ 
    converge a.s and,
    thus, the summand $\norm{\grad f(v^k)}^2$ vanishes as $k \to \infty$ a.s. 
    For the convergence rate, 
    we take the expectation in \eqref{eq: zeroth order proof tech}, 
    which gives 
    \begin{align*}
        & \min_{k = 0,\ldots, n} \EE[\norm{\grad f(v^k)}^2]
        \le \frac{1}{n+1} \sum_{k=0}^n \EE[\norm{\grad f(v^k)}^2] \\
        & \qquad \le \frac{8 (d+4) L}{n+1} \sum_{k=0}^n \bigg[ \EE[\mathcal R(A,B) -  f(v^k) ] - \EE[\mathcal R(A,B) -  f(v^{k+1})] + C \mu_k^2 \bigg] \\
        & \qquad = \frac{8 (d+4) L}{n+1} \bigg[ \EE[ \mathcal R(A,B)- f(v^0) ] - \EE[\mathcal R(A,B) - f(v^{K})] + C \sum_{k=0}^n \mu_k^2 \bigg] \\
        & \qquad \le \frac{8 (d+4) L}{n+1} \bigg[ \mathcal R(A,B)- f(v^0) + C \sum_{k=0}^\infty \mu_k^2 \bigg],
    \end{align*}
    where in the last step we used $\mathcal R(A,B) - f(v^{K}) \ge 0$.
    
    Note that $\sum_{k = 0}^{\infty} \tfrac{\nu}{4} \norm{\grad f(v^k)}^2 < \infty$ 
    also implies a similar a.s. convergence rate 
    \[
        \min_{k = 0,\ldots, n} \norm{\grad f(v^k)}^2
        \le \frac{1}{n+1} \sum_{k=0}^n \norm{\grad f(v^k)}^2
        \le \frac{1}{n+1} \sum_{k=0}^\infty \norm{\grad f(v^k)}^2,
        \ \text{a.s.}
    \]
    Yet, the series' value 
    is a random variable and its dependence on $d$ and $L$ is unclear. 
\end{proof}

\noindent
\textbf{Lemma~\ref{l: direction main}}
    For a fixed $v \in  \sphere_B^{d-1}$  and $\tilde X \sim \cN(0, I_d)$, the random variable
    $X \coloneq P_v \tilde X/\|P_v \tilde X\|$
    is uniformly distributed 
    on $T_{v} \cap \sphere^{d-1} \simeq \sphere^{d-2}$.
    Moreover, it holds
    \begin{align*}
        \EE_{x \sim X}[x x^\tT] 
       = \frac{1}{d-1} P_{v} = \frac{1}{d-1} \Bigl(I_d - \tfrac{Bv}{\norm{Bv}}\tfrac{(Bv)^\tT}{\norm{Bv}}\Bigr) .
    \end{align*}
\begin{proof}
    For an orthonormal basis $\{ u_i \in \RR^d: i=1,\ldots,d-1 \}$ of $T_v$ and $u_d \coloneqq Bv / \norm{Bv} \in T_v^\perp$, set  
    $U \coloneqq ( u_1 \,  \ldots \, u_d) = (U_v \, u_d)$.
    Then it holds
    \[
        \tilde X 
        = U U^\tT \tilde X = U \tilde Z , 
        \quad \tilde Z \coloneqq U^\tT \tilde X \sim \cN(0, I_d).
    \]
    Then $P_v \tilde X = U_v Z$ 
    with $Z \coloneqq (\tilde Z_1, \ldots, \tilde Z_{d-1})^\tT \sim \cN(0, I_{d-1})$ 
    and we know by \cite[Ex.~3.3.7]{Vershynin_2018} 
    that $Y \coloneqq Z / \norm{Z}$ follows a uniform distribution 
    on $\sphere^{d-2}$.
    Finally, since
    \begin{equation} \label{eq: x distribution}
        X = \frac{P_{v} \tilde X}{\norm{P_{v} \tilde X}}
        = \frac{U_v Z}{\norm{U_v Z}} = U_v \frac{Z}{\norm{Z}} = U_v Y
        \end{equation}
    and $U_{v} \colon \RR^{d-1} \to T_v$ is an isometry, 
    this yields the first assertion.
    Next, we use the representation $X = U_v Y$ to compute the covariance of $X$. Since $Y$ is uniform distribution on $\sphere^{d-2}$ and does not depend on $v$, we get
    \begin{align*}
        \EE_{x\sim X}[x x^\tT ] 
        & = \EE_{y \sim Y}[U_{v} y (U_{v} y)^\tT ] 
        = U_{v} \, \EE_{y \sim Y}[ y y^\tT] \, U_{v}^\tT.
    \end{align*}
    The random vector $\sqrt{d-1} \, Y$ is isotropic, 
    see \cite[Def.~3.2.1, Ex.~3.3.1]{Vershynin_2018}, 
    which means that the covariance of $\EE_{y \sim Y}[ y y^\tT] = \frac{1}{d-1} I_{d-1}$.
    Furthermore, 
    by the properties of the orthogonal projection, we have
    \[
        P_{v} w = \sum_{i = 1}^{d-1} \scp{w}{u_i}u_i =  U_{v}  U_{v}^\tT w,
        \quad \text{for all} \quad w \in \RR^{d}.
    \]
    Consequently,
    \[
        \EE_{x\sim X}[x x^\tT ] 
        = \tfrac{1}{d-1}U_{v} I_{d-1} U_{v}^\tT = \tfrac{1}{d-1} P_{v}.
        \qedhere
    \]
\end{proof}

\noindent
\textbf{Lemma~\ref{l: direction main_iii}}
 Let $M$ be an affine subspace in $\RR^d$ of dimension $r$,
    and     $$\varphi : M \setminus\{0\} \to \sphere^{d-1}, x \mapsto x/\norm{x}.$$
    \begin{itemize}
        \item[{\textrm i)}]
        If $0 \in M$ and $r < d$, then $\varphi(M)$ is of measure zero  with respect to  the surface measure $\sigma_{\sphere^{d-1}}$.
      \item[{\textrm ii)}]  
    If $0 \not \in M$ and $r < d-1$, then
        $\varphi(M)$ is of measure zero 
        with respect to $\sigma_{\sphere^{d-1}}$.
        \end{itemize}

\begin{proof}
For i) If $0 \in M$,
    then $M$ is an linear subspace of dimension $r$ in $\RR^d$,
    and $\varphi(M \setminus\{0\}) = M \cap \sphere^{d-1}$
    is a $r-1$-dimensional submanifold. It has measure zero in $\sphere^{d-1}$ 
    if $r-1 < d-1$, cf.~\cite[Cor.~6.12]{Lee2012smoothmanifolds}.
\\
For ii)
If $0 \not \in M$,
    then $M = m + W$,
    where $W$ is a linear subspace of dimension $r$ and $m \not\in W$.
    then $\varphi : M \to \sphere^{d-1}$ is a smooth function on $M$
    and $\varphi(M)$ is a submanifold of dimension $r$.
    By \cite[Prop.~6.5]{Lee2012smoothmanifolds},
    $\varphi(M)$ is of measure zero in $\sphere^{d-1}$
    if $r < d-1$.
\end{proof}

\noindent
\textbf{Lemma~\ref{l:uniform_BS}}
    Let $\varepsilon > 0$ 
    and $v, v^* \in \sphere_B^{d-1}$.
    Consider
    \begin{align*}
        \mathcal D_v 
        \coloneqq 
        \left\{ x \in T_{v} \cap \sphere^{d-1}
        \;\middle|\;
        \;\exists \tau \in \RR : 
        R_{v}(\tau x)\in \mathbb B_{\varepsilon,B}(\pm v^*) \cap \sphere_B^{d-1}\right\},
    \end{align*}
    where $\mathbb B_{\varepsilon,B}(\pm v^*) \coloneqq \mathbb B_{\varepsilon,B}(v^*) \cup \mathbb B_{\varepsilon,B}(-v^*)$
    is the union of the $\varepsilon$-balls with respect to $\normB{\cdot}$
    around $v^*$ and $-v^*$. 
    Then, for the surface measure of the unit sphere $\sigma_{Bv}$ in $T_{v}$,
    there exists $p = p(\varepsilon, d, B) > 0$ such that 
    \begin{align*}
        \inf_{v \in \sphere_B^{d-1}} 
        \; \sigma_{Bv}(\mathcal D_{v}) \geq p.
    \end{align*}

\begin{figure}
    \includegraphics[width=\linewidth, clip=true, trim=10pt 18pt 40pt 8pt]{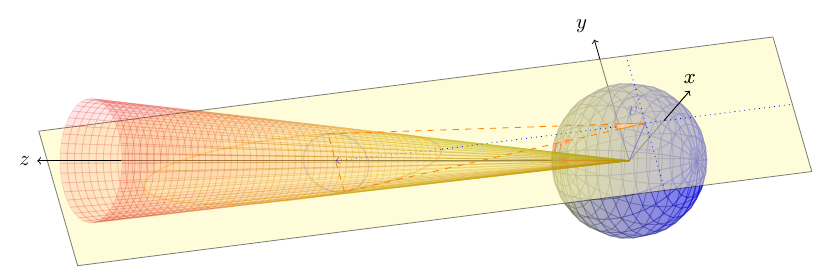}
    \caption{%
    Visualization in the three-dimensional case of the construction in the proof of Lemma~\ref{l:uniform_BS}. Consider a cone formed by $\tilde u \in \mathbb B_{\varepsilon,B}(v^*) \cap \sphere_B^{d-1}$ for $v^*$ aligned with $z$-axis in the figure. Its intersection with the tangent space $T_v$ forms an ellipse, which contains a Euclidean ball $\mathbb B_R(z)$. The retraction of points in this ball lies in $B_{\varepsilon,B}(v^*)$. Therefore, we look at the cone $\mathcal K$ of vectors in $T_v$ that point to $\mathbb B_R(z)$, defined by an angle $\alpha$ (in red). We show that $\sin(\alpha)$ is uniformly bounded from below for all $v \in \sphere_B^{d-1}$ by a positive constant that yields a desired probability $p$.
    }
    \label{fig:sphere_tangent_cone}
\end{figure}

\begin{proof}
    We first proof the bound for the set 
    \begin{align*}
    \tilde{\mathcal D}_{B,v}   
        \coloneqq 
        \left\{ x \in T_v \cap \sphere^{d-1}
        \;\middle|\;
        \;\exists \tau \in \RR : 
        \tfrac{v + \tau x}{\sqrt{1 + \tau^2 \normB{x}^2}} \in \cS \cap \sphere_B^{d-1}\right\} \subseteq \mathcal D_{B,v},
    \end{align*}
    where $\cS \subseteq \mathbb B_{\varepsilon,B}(\pm v^*)$. Then,
    \[
    \inf_{v \in \sphere_B^{d-1}} 
        \; \sigma_{Bv}(\mathcal D_{B,v})
        \ge \inf_{v \in \sphere_B^{d-1}} 
        \; \sigma_{Bv}(\tilde{\mathcal D}_{B,v}) \ge p_{\varepsilon,d,B}.
    \]
    In the following, without loss of generality, 
    we construct $\cS$ for $\scpB{v}{v^*} \geq 0$ and otherwise, 
    it can be done analogously by replacing $v^*$ with $-v^*$. Therefore, we have
    \[
    \normB{v - v^*}^2 = \normB{v}^2 + \normB{v^*}^2 - 2 \scpB{v}{v^*} \le 2.
    \]
    
    If $\varepsilon \ge \normB{v - v^*}$, it suffices to take $\cS = \{v\}$, 
    which satisfies $\cS \subset \mathbb B_{\varepsilon,B}$ 
    and for all $x \in T_v \cap \sphere^{d-1}$ and $\tau = 0$ we get
    \[
    	\frac{v+ \tau x}{\normB{v+ \tau x}}
	= \frac{v+ \tau x}{\sqrt{1 + \tau^2 \normB{x}^2}} = v \in \cS,
    \]
    so that $\sigma_{Bv}(\tilde{\mathcal D}_{B,v}) = 1$.

    Thus, we focus on the case $\varepsilon < \normB{v - v^*} \le \sqrt{2}$. 
    In this case, we take $\cS = \mathbb B_{r,B}(\tilde u)$ 
    where $r \coloneqq \varepsilon^2 / 8$,  $c \coloneqq \varepsilon / 4 \sqrt 2$ and
    \begin{equation}\label{eq: def u appendix}
        \tilde u \coloneqq \frac{u}{\normB{u}}
        \quad \text{with} \quad
        u \coloneqq v^* - c(v^* - v) 
        = (1 - c) v^* + c v.
    \end{equation}
    Next, we step-by-step show that this choice of $\cS$ satisfies all desired properties. Since $B \in \pd$, the norm $\normB{\cdot}$ is strongly. By definition, $u$ is a convex combination of $v$ and $v^*$ giving $\normB{u} < 1$. Furthermore,
    \begin{equation*}
        \normB{\tilde u - u}
        = \bigg\| \frac{u}{\normB{u}} - u \bigg\|_B
        = \normB{u}\left(\frac{1}{\normB{u}} - 1\right)
        = 1 - \normB{u}.
    \end{equation*}
    We bound
    \begin{align*}
        \normB{u} \geq \norm{v^*} - \normB{u - v^*} 
        \stackrel{\eqref{eq: def u appendix}}{=} 1 - c \normB{v - v^*}
        \geq 1 - \sqrt 2 c,
    \end{align*}
    which yields
    \begin{align*}
        \normB{\tilde u - v^*} 
        \leq 
        \normB{u - v^*} + \normB{\tilde u - u}
        \leq 
        \sqrt 2 c + 1 - \normB{u}
        \leq 2 \sqrt 2 c = \tfrac{\varepsilon}{2}.
    \end{align*}
    and $\cS = \mathbb B_{r,B}(\tilde u) \subset \mathbb B_{\varepsilon}(v^*)$.

    For $\tilde u$, 
    by \eqref{eq: def u appendix} and $\normB{u} < 1$ we have
    \begin{align}
        \label{eq:lower_bound_inner_product}
        \scpB{\tilde u}{v}
        & =
        \frac{\scpB{u}{v}}{\normB{u}}
        = 
        \frac{\scpB{v^*(1 - c) + cv}{v}}{\normB{u}} \notag \\
        &= 
        \frac{1}{\normB{u}}((1 - c)\scpB{v^*}{v} + c)
        \geq 
        \frac{c}{\normB{u}}
        \geq 
        c > 0.
    \end{align}
    Thus, we define $z \coloneqq \frac{1}{\scpB{\tilde u}{v}} \tilde u - v \neq 0$ satisfying
    \[
        \scp{z}{Bv} 
        =  \scpB{z}{v} 
        = \frac{\scpB{\tilde u}{v}}{\scpB{\tilde u}{v}}  - \scpB{v}{v} 
        = 1 - 1 
        = 0,
    \]
    so that $z \in T_v$.  
    Let us consider $\mathbb B_{R}(z)$, a Euclidean ball of radius $R \coloneqq r / 2 \norm{B}^{1/2} \scpB{\tilde u}{v}$. 
    For all points $z + w \in \mathbb B_{R}(z) \cap T_v$ we have
    \begin{align*}
        \scpB{\tilde u}{v}(v + z + w)  = \tilde u + \scpB{\tilde u}{v} w.
    \end{align*}
    and
    \[
        \frac{v + z + w}{\sqrt{1 + \normB{z +w}^2}} 
        = \frac{v + z + w}{\normB{v + z + w}} 
        = \frac{\tilde u + \scpB{\tilde u}{v} w}{\normB{\tilde u + \scpB{\tilde u}{v} w}}. 
    \]
    Using that for all $x,y \in \RR^d$ with $\normB{y}=1$ it holds
    \begin{align*}
        \bigg\|\frac{x}{\normB{x}} - y\bigg\|_B
        & \le \bigg\| \frac{x}{\normB{x}} - x \bigg\|_B + \normB{x - y} 
        = \bigg| \frac{1}{\normB{x}} - 1 \bigg| \normB{x}  + \normB{x - y}\\
        & = | \normB{x} - \normB{y} | + \normB{x - y}
        \le 2 \normB{x - y}
    \end{align*}
    we obtain
    \[
        \bigg\| \frac{v + z + w}{\sqrt{1 + \normB{z +w}^2}} - \tilde u \bigg\|_B
        \le 2 \normB{\tilde u + \scpB{\tilde u}{v} w - \tilde u}
        \le 2 \scpB{\tilde u}{v} \norm{B}^{1/2} R
        \le r. 
    \]
    Hence, 
    $R_v(w) \in \cS$ for all $w\in \mathbb B_{R}(z) \cap T_v$
    If $0 \in B_{R}(z) \cap T_v$, 
    then for all $x \in \sphere^{d-1} \cap T_v$ we can find $\tau \in \RR$ 
    such that $\tau x \in \mathbb B_{R}(z) \cap T_v$
    and $\sigma_{Bv}(\tilde{\mathcal D}_{B,v}) = 1$.
    Otherwise, 
    we consider a hyperspherical cap 
    \[
        \mathcal K = \{ x \in \sphere^{d-1} \cap T_v \mid \exists \tau \in \RR
        \ \text{such that} \ 
        \tau x \in \mathbb B_{R}(z) \cap T_v \} \subseteq \tilde{\mathcal D}_{B,v}
    \]
    It is defined by an angle $\alpha$ with 
    \[
        \sin(\alpha) 
        = \frac{R}{\norm{z}} 
        = \frac{\varepsilon^2}{16 \norm{B}^{1/2} \scpB{\tilde u}{v}\norm{z}}.
    \]
    By \eqref{eq:lower_bound_inner_product} we have
    \begin{align*}
        \scpB{\tilde u}{v}^2 \norm{z}^2 
        & = \norm{\tilde u - \scpB{\tilde u}{v} v}^2
        \le \lambda_d(B)^{-1} \normB{\tilde u - \scpB{\tilde u}{v} v}^2 \\
        & = \lambda_d(B)^{-1}(1 - \scpB{\tilde u}{v}^2) \le \lambda_d(B)^{-1}(1-c^2) 
    \end{align*}
    giving 
    \[
        \sin(\alpha) 
        \ge \frac{\varepsilon^2}{16 \kappa(B)^{1/2} \sqrt{1- \varepsilon^2 / 32}} 
        \eqqcolon \sin(\beta_{\varepsilon,B}).
    \]
    Then, 
    the hyperspherical cap $\mathcal K_{\varepsilon,B}$ 
    defined by the angle $\beta_{\varepsilon,B}$ is the subset of $\mathcal K$. 
    By \cite{li2011concise}, 
    its normalized area is given by the incomplete Beta function
    $\mathrm{B}(x, \alpha,\beta) \coloneqq \int_0^x t^{\alpha-1}(1-t)^{\beta-1} \dx t$
    \[
        \sigma_{Bv}(\mathcal K_{\varepsilon,B}) 
        = \mathrm{B}(\sin(\beta_{\varepsilon,B}),\tfrac{d-2}{2}, \tfrac{1}{2})
        = \int_{0}^{\sin(\beta_{\varepsilon,B})} t^{\frac{d}{2}-2} (1 - t)^{-\frac{1}{2}} \dx t \eqqcolon p_{\varepsilon,d,B}
    \]
    Therefore, we get
    \[
        \sigma_{Bv}(\tilde{\mathcal D}_{B,v}) 
        \ge \sigma_{Bv}(\mathcal K)
        \ge \sigma_{Bv}(\mathcal K_{\varepsilon,B})
        = p_{\varepsilon,d,B}.
        \qedhere
    \]
\end{proof}

\noindent
\textbf{Lemma~\ref{lem:bound_B_kappa}}
    Let $B \in \pd, v \in \sphere_B^{d-1}$ and $y \in \RR^d$.
    Then we have 
    \begin{equation*}
        \norm{(I_d - Bvv^\tT)y} \leq \sqrt{\kappa(B)} \norm{(I_d - \tfrac{Bv (Bv)^\tT}{\norm{Bv}^2}) y}.
    \end{equation*}
    
\begin{proof}
    We define $u \coloneqq Bv / \norm{Bv} \in \sphere^{d-1}$,
    such that 
    \begin{equation*}
        1 = \normB{v} = \scp{v}{Bv} = \scp{v}{u}\norm{Bv}    
    \end{equation*}
    and $\scp{v}{u} = \norm{Bv}^{-1}$.
    Then, we get
    \begin{equation*}
        B vv^\tT u = Bv \scp{u}{v} = \tfrac{Bv}{\norm{Bv}} = u.
    \end{equation*}
    We decompose any $y = \alpha u + P_{v} y$,
    for some $\alpha \in \RR$,
    which yields
    \begin{align*}
        (I_d - B v v^\tT) y
        & = \alpha (I_d - B v v^\tT) u  
        + (I_d - B v v^\tT) P_{v} y \\
        & = \alpha (u - B vv^\tT u) + (I_d - B v v^\tT) P_{v} y
        = (I_d - B v v^\tT) P_{v} y,
    \end{align*}
    and hence
    \begin{align}\label{eq: proj bound tech}
        \norm{(I_d - B v v^\tT) y}
        \leq \sup_{\substack{w \in T_v \\ \norm{w} = 1}}
        \norm{(I_d - B v v^\tT) w} \cdot \norm{P_{v} y}.
    \end{align}
    Expanding the norm in the supremum gives
    \begin{align*}
        \sup_{\substack{w \in T_v \\ \norm{w} = 1}}
        \norm{(I_d - B v v^\tT) w}^2
        &= \sup_{\substack{w \in T_v \\ \norm{w} = 1}}
        \norm{w}^2 - 2 \scp{v}{w} \scp{B v}{w} + \scp{v}{w}^2 \norm{B v}^2 \\
        &= \sup_{\substack{w \in T_v \\ \norm{w} = 1}} 
        1 + \scp{v}{w}^2 \norm{B v}^2.
    \end{align*}
    Since $\scp{v}{w}^2 = \scp{P_{v}v}{w}^2$ for all $w \in T_v$, the supremum is attained at $w = P_{v} v/ \norm{P_{v}v}$ yielding
    \begin{align*}
        \sup_{\substack{w \in T_v \\ \norm{w} = 1}}
        \scp{v}{w}^2
        &  = \norm{P_{v}v}^2
        = \norm{v - \scp{v}{u}u}^2
        = \norm{v}^2 - \scp{v}{u}^2
        = \norm{v}^2 - \tfrac{1}{\norm{B v}^2}.
    \end{align*}
    Hence,
    we have 
    \begin{align}\label{eq: proj bound tech 2}
        \sup_{\substack{w \in T_v \\ \norm{w} = 1}}
        \norm{(I_d - B v v^\tT) w}^2
        = \underbracket{1 + \norm{B v}^2 (\norm{v}^2 - \tfrac{1}{\norm{B v}^2})}_{= \norm{v}^2\norm{B v}^2}
        \leq \kappa(B).
    \end{align}
    Combining \eqref{eq: proj bound tech} and \eqref{eq: proj bound tech 2} gives the assertion.
\end{proof}

\section{Extension to the complex case}\label{sec: complex}

In this appendix, we extend our method to complex matrices. 
That is, our goal is to find the maximum of the generalized real Rayleigh quotient
\begin{equation}
    \label{eq:cgRq}
    \mathcal R(A, B) 
    = \max_{v \in \CC^d\setminus\{0\}} \RE[ r(A, B, v) ]
    = \max_{v \in \CC^d\setminus\{0\}} \frac{\RE[\scp{v}{Av}]}{\scp{v}{B v}}
\end{equation}
for complex vectors and complex inner product
$\scp{x}{y} = \sum_{i=1}^d y_i \bar x_i$. Here $A \in \CC^{d\times d}$
and $B \in \CC^{d\times d}$ is a \textit{Hermitian} positive definite matrix.
Recall that a matrix is Hermitian if $B^* = B$ 
with conjugate transpose $B^* \coloneqq \bar B^\tT$.
Just as in the real case, 
maximizing \eqref{eq:cgRq} is equivalent to finding the leading real eigenvalue 
of $B^{-1}\herA$ with $\herA \coloneqq \tfrac{1}{2}(A + A^*)$ \cite{trefethen2005spectrapseudospectra}.

To apply our algorithms for the complex problem \eqref{eq:cgRq}, 
we identify a complex vector $v \in \CC^{d}$ with real vector $\tilde v = (\RE v, \IM v) \in \RR^{2d}$ and a matrix $M \in \CC^{d \times d}$ with 
$\tilde M \in \RR^{2d \times 2d}$ so that $\widetilde{Mv} = \tilde M \tilde v$ for all $v \in \CC^d$. Then, for all $x, v \in \CC^d$ and $M \in \CC^{d \times d}$ we have
\[
\langle \tilde x, \tilde M \tilde v \rangle = \RE[\scp{x}{Mv}]
\]
and 
\[
\mathcal R(A, B) 
= \max_{v \in \CC^d\setminus\{0\}} \frac{\RE[\scp{v}{Av}]}{\scp{v}{Bv}}
= \max_{\tilde v \in \RR^{2d}\setminus\{0\}} \frac{ \langle \tilde v, \tilde A \tilde v \rangle}{ \langle \tilde v, \tilde B \tilde v \rangle}.
\]
Consequently, 
we compute $\tilde v \in \sphere^{2d-1}_B$ 
by applying Algorithms~\ref{alg:gRq} or~\ref{alg:gRq m-sample} 
for $(\tilde A, \tilde B)$ and constructing the corresponding $v \in \CC^d$. 
Moreover, the convergence guarantees extend to the complex case.  

\end{document}